\documentclass[12pt,reqno]{amsart}
\usepackage{amsmath}
\usepackage{amsfonts}
\usepackage{mathrsfs}
\usepackage {amssymb,euscript}
\usepackage {amsmath}
\usepackage {amsthm}
\usepackage {amscd}
\usepackage{geometry}
\usepackage{hyperref}
\usepackage{color}
\usepackage{epsfig}
\usepackage{caption}

\usepackage{tikz}
\usepackage{tkz-euclide}
\usetkzobj{all}

\usepackage{graphics}

\title[Polynomial Blow-Up Upper Bounds]{Polynomial Blow-up Upper Bounds for\\  the Einstein-scalar field System\\ Under Spherical Symmetry}
\date{\today}

\author{Xinliang An}
\address{Department of Mathematics, National University of Singapore,
10 Lower Kent Ridge Road, Singapore, 119076}
\email{matax@nus.edu.sg}

\author{Ruixiang Zhang}
\address{Department of Mathematics, University of Wisconsin-Madison,
480 Lincoln Drive, Madison, WI, USA, 53706}
\email{ruixiang@math.wisc.edu}

\theoremstyle{definition}
\newtheorem{lemma}{Lemma}[section]

\newtheorem{proposition}[lemma]{Proposition}
\newtheorem{theorem}[lemma]{Theorem}

\newtheorem{remark}{Remark}

\numberwithin{equation}{section}

\begin{document}

\newcommand{\ub}{\underline{u}}
\newcommand{\Cb}{\underline{C}}
\newcommand{\Lb}{\underline{L}}
\newcommand{\Lh}{\hat{L}}
\newcommand{\Lbh}{\hat{\Lb}}
\newcommand{\phib}{\underline{\phi}}
\newcommand{\Phib}{\underline{\Phi}}
\newcommand{\Db}{\underline{D}}
\newcommand{\Dh}{\hat{D}}
\newcommand{\Dbh}{\hat{\Db}}
\newcommand{\omb}{\underline{\omega}}
\newcommand{\omh}{\hat{\omega}}
\newcommand{\ombh}{\hat{\omb}}
\newcommand{\Pb}{\underline{P}}
\newcommand{\chib}{\underline{\chi}}
\newcommand{\chih}{\hat{\chi}}
\newcommand{\chibh}{\hat{\chib}}

\newcommand{\alb}{\underline{\alpha}}
\newcommand{\zeb}{\underline{\zeta}}
\newcommand{\beb}{\underline{\beta}}
\newcommand{\etb}{\underline{\eta}}
\newcommand{\Mb}{\underline{M}}
\newcommand{\oth}{\hat{\otimes}}

\def\a {\alpha}
\def\b {\beta}
\def\ab {\alphab}
\def\bb {\betab}
\def\nab {\nabla}

\def\ub {\underline{u}}
\def\th {\theta}
\def\Lb {\underline{L}}
\def\Hb {\underline{H}}
\def\chib {\underline{\chi}}
\def\chih {\hat{\chi}}
\def\chibh {\hat{\underline{\chi}}}
\def\omegab {\underline{\omega}}
\def\etab {\underline{\eta}}
\def\betab {\underline{\beta}}
\def\alphab {\underline{\alpha}}
\def\Psib {\underline{\Psi}}
\def\hot{\widehat{\otimes}}
\def\Phib {\underline{\Phi}}
\def\thb {\underline{\theta}}
\def\t {\tilde}
\def\st {\tilde{s}}

\def\rhoc{\check{\rho}}
\def\sigmac{\check{\sigma}}
\def\Psic{\check{\Psi}}
\def\kappab{\underline{\kappa}}
\def\betabc {\check{\underline{\beta}}}

\def\d {\delta}
\def\f {\frac}
\def\i {\infty}
\def\l {\bigg(}
\def\r {\bigg)}
\def\S {S_{u,\underline{u}}}
\def\o{\omega}
\def\O{\Omega}\
\def\be{\begin{equation}\begin{split}}
\def\en{\end{split}\end{equation}}
\def\at{a^{\frac{1}{2}}}
\def\af{a^{\frac{1}{4}}}
\def\od{\omega^{\dagger}}
\def\ombd{\underline{\omega}^{\dagger}}
\def\K{K-\frac{1}{|u|^2}}
\def\ut{\frac{1}{|u|^2}}
\def\s{\frac{\delta a^{\frac{1}{2}}}{|u|}}
\def\Kb{K-\frac{1}{(u+\underline{u})^2}}
\def\bf{b^{\frac{1}{4}}}
\def\bt{b^{\frac{1}{2}}}
\def\de{\delta}
\def\ls{\lesssim}
\def\om{\omega}
\def\Om{\Omega}

\newcommand{\e}{\epsilon}
\newcommand{\et} {\frac{\epsilon}{2}}
\newcommand{\ef} {\frac{\epsilon}{4}}
\newcommand{\LH} {L^2(H_u)}
\newcommand{\LHb} {L^2(\underline{H}_{\underline{u}})}
\newcommand{\M} {\mathcal}
\newcommand{\TM} {\tilde{\mathcal}}
\newcommand{\p}{\psi\hspace{1pt}}
\newcommand{\q}{\underline{\psi}\hspace{1pt}}
\newcommand{\Li}{_{L^{\infty}(S_{u,\underline{u}})}}
\newcommand{\Lt}{_{L^{2}(S)}}
\newcommand{\da}{\delta^{-\frac{\epsilon}{2}}}
\newcommand{\db}{\delta^{1-\frac{\epsilon}{2}}}
\newcommand{\D}{\Delta}


\renewcommand{\div}{\mbox{div }}
\newcommand{\curl}{\mbox{curl }}
\newcommand{\trchb}{\mbox{tr} \chib}
\def\trch{\mbox{tr}\chi}
\newcommand{\tr}{\mbox{tr}}

\newcommand{\Ls}{{\mathcal L} \mkern-10mu /\,}
\newcommand{\eps}{{\epsilon} \mkern-8mu /\,}

\newcommand{\xib}{\underline{\xi}}
\newcommand{\psib}{\underline{\psi}}
\newcommand{\rhob}{\underline{\rho}}
\newcommand{\thetab}{\underline{\theta}}
\newcommand{\gammab}{\underline{\gamma}}
\newcommand{\nub}{\underline{\nu}}
\newcommand{\lb}{\underline{l}}
\newcommand{\mub}{\underline{\mu}}
\newcommand{\Xib}{\underline{\Xi}}
\newcommand{\Thetab}{\underline{\Theta}}
\newcommand{\Lambdab}{\underline{\Lambda}}
\newcommand{\vphb}{\underline{\varphi}}

\newcommand{\ih}{\hat{i}}

\newcommand{\tcL}{\widetilde{\mathscr{L}}}

\newcommand{\sRic}{Ric\mkern-19mu /\,\,\,\,}
\newcommand{\sL}{{\cal L}\mkern-10mu /}
\newcommand{\sLh}{\hat{\sL}}
\newcommand{\sg}{g\mkern-9mu /}
\newcommand{\seps}{\epsilon\mkern-8mu /}
\newcommand{\sd}{d\mkern-10mu /}
\newcommand{\sR}{R\mkern-10mu /}
\newcommand{\snab}{\nabla\mkern-13mu /}
\newcommand{\sdiv}{\mbox{div}\mkern-19mu /\,\,\,\,}
\newcommand{\scurl}{\mbox{curl}\mkern-19mu /\,\,\,\,}
\newcommand{\slap}{\mbox{$\triangle  \mkern-13mu / \,$}}
\newcommand{\sGamma}{\Gamma\mkern-10mu /}
\newcommand{\somega}{\omega\mkern-10mu /}
\newcommand{\somb}{\omb\mkern-10mu /}
\newcommand{\spi}{\pi\mkern-10mu /}
\newcommand{\sJ}{J\mkern-10mu /}
\renewcommand{\sp}{p\mkern-9mu /}
\newcommand{\su}{u\mkern-8mu /}

\begin{abstract}
For general gravitational collapse, inside the black-hole region, singularities $(r=0)$ may arise. In this article, we aim to answer how strong these singularities could be. {\color{black} We analyse the behaviours of various geometric quantities. In particular, we show that in the most singular scenario, the Kretschmann scalar obeys polynomial blow-up upper bounds $O(1/r^N)$.} This improves previously best-known double-exponential upper bounds $O\big(\exp\exp(1/r)\big)$. Our result is sharp in the sense that there are known examples showing that no sub-polynomial upper bound could hold. Finally we do a case study on perturbations of the Schwarzschild solution.
\end{abstract}

\maketitle

\section{Introduction}
\subsection{Motivation}
In \cite{Chr.1}, Christodoulou studied the dynamical evolution of Einstein-scalar field system:
\begin{equation}\label{ES}
\begin{split}
&\mbox{Ric}_{\mu\nu}-\f12Rg_{\mu\nu}=2T_{\mu\nu},\\
&T_{\mu\nu}=\partial_{\mu}\phi \partial_{\nu}\phi-\f12g_{\mu\nu}\partial^{\sigma}\phi \partial_{\sigma}\phi.
\end{split}
\end{equation}
{\color{black}Since $\nab^{\mu}(\mbox{Ric}_{\mu\nu}-\f12Rg_{\mu\nu})=0$, the scalar field satisfies $\Box_g \phi=0$.}

Under spherical symmetry, Christodoulou first established a sharp trapped surface\footnote{A trapped surface is a two-dimensional sphere, with both incoming and outgoing null expansions negative.
} formation criterion. {\color{black}Consider the characteristic initial value problem for (\ref{ES}) in the rectangle region of a Penrose diagram blow:}

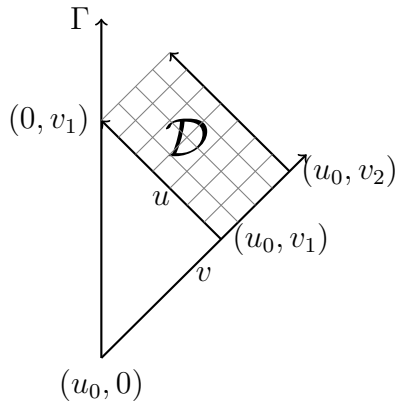
\begin{minipage}[!t]{0.4\textwidth}
	\begin{tikzpicture}[scale=0.9]
	\node[] at (1.25,3.25) {\LARGE $\mathcal{D}$};
	\begin{scope}[thick]
	\draw[->] (0,0) node[anchor=north]{$(u_0,0)$} -- (0,5)node[anchor = east]{$\Gamma$};
	\draw[->] (0,0) --node[anchor = north]{$v$} (3,3);
	\draw[->] (1.75,1.75) node[anchor=west]{$(u_0,v_1)$} --node[anchor=north]{$u$} (0,3.5)node[anchor = east]{$(0,v_1)$};
	\draw[->] (2.75,2.75) node[anchor=west]{$(u_0,v_2)$} -- (1,4.5);
	\end{scope}
	\begin{scope}[gray]
	\draw (2,2) -- (0.25,3.75);
	\draw (2.25,2.25) -- (0.5,4);
	\draw (2.5,2.5) -- (0.75,4.25);
	\draw(1.5,2) -- (2.5,3);
	\draw(1.25,2.25) -- (2.25,3.25);
	\draw(1,2.5) -- (2,3.5);
	\draw(0.75,2.75) -- (1.75,3.75);
	\draw(0.5,3) -- (1.5,4);
	\draw(0.25,3.25) -- (1.25,4.25);
	\draw(0,3.5) -- (1,4.5);
	\end{scope}
	\end{tikzpicture}
\end{minipage}
\begin{minipage}[!t]{0.58\textwidth}
We use a double-null foliation. Here $u$ and $v$ are optical functions:  $u=\mbox{{\color{black}constant}}$ stands for the outgoing null hypersurface; $v=\mbox{{\color{black}constant}}$ stands for the incoming null hypersurface. \\

\noindent Under spherical symmetry, axial $\Gamma$ is the center (invariant under $SO(3)$). Initial data are prescribed along outgoing cone $u=u_0$ and incoming cone $v=v_1$.
\end{minipage}
\hspace{0.05\textwidth}	

\noindent Under the above assumption, we have the following ansatz for the metric of the $3+1$-dimensional spacetime: 
\begin{equation}\label{metric}
g_{\mu\nu}dx^{\mu}dx^{\nu}=-\O^2(u,v)dudv+r^2(u,v)\big(d\theta^2+\sin^2\theta d\phi^2\big).
\end{equation}
\noindent Each point $(u,v)$ in above diagram stands for a $2$-sphere $S_{u,v}$. We define its Hawking mass as 
\begin{equation}\label{Hawking mass}
m(u,v)=\f{r}{2}(1+4\O^{-2}\partial_u r \partial_v r).
\end{equation}
{\color{black}\noindent For initial mass input along $u=u_0$, we define
$$\eta_0:=\f{m(u_0, v_2)-m(u_0, v_1)}{r(u_0, v_2)}, \mbox{ and denote } \d_0:=\f{r(u_0, v_2)-r(u_0, v_1)}{r(u_0, v_2)}.$$

\begin{theorem}\label{Christodoulou theorem spherical}
	Let
	\begin{align*}
	E(x):=\frac{x}{(1+x)^2}\bigg[\ln\bigg(\frac{1}{2x}\bigg)+5-x\bigg].
	\end{align*}
	We prescribe characteristic initial data along $u=u_0$ and $v=v_1$ for solving (\ref{ES}). For initial mass input along $u=u_0$,  suppose that the following lower bound holds for $\eta_0$:
	\begin{align*}
	\eta_0>E(\delta_0)
	\end{align*}
		Then there exist a trapped surface i.e. $\partial_vr< 0$ in $\mathcal{D}$.
		\end{theorem}
\begin{remark}
	By comparing the order of the lower bounds of $\eta_0$, for $0<\d_0\ll 1$ we have: If $\eta_0>\delta_0\ln\bigg(\frac{1}{\delta_0}\bigg)$, then a trapped surface is guaranteed to form in $\mathcal{D}$.
\end{remark}
\begin{remark}
To prove Theorem \ref{Christodoulou theorem spherical}, Christodoulou didn't impose any assumption along incoming cone $v=v_1$. And his original proof was based on a geometric Bondi coordinate together with a null frame. In a forth coming paper \cite{An-Lim} we reprove Theorem \ref{Christodoulou theorem spherical} with double null foliations and generalize this result to Einstein-Maxwell-scalar field system.  
\end{remark}

\noindent Once a trapped surface is formed, in \cite{Chr.1} Christodoulou further showed that the {\color{black}Penrose} {\color{black}diagram} for such spacetimes {\color{black}is} as follows:

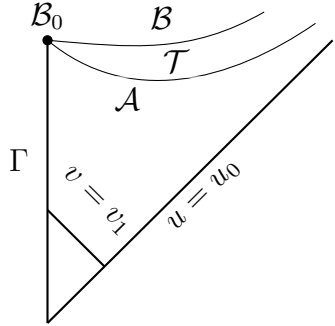
\begin{center}
\begin{minipage}[!t]{0.4\textwidth}
\begin{tikzpicture}[scale=0.75]
\draw [white](-1, -2.5)-- node[midway, sloped, above,black]{$\Gamma$}(0, -2.5);
\draw [white](0, 0)-- node[midway, sloped, above,black]{$\mathcal{B}$}(4, 0);
\draw [white](0, -0.75)-- node[midway, sloped, above,black]{$\mathcal{T}$}(4.5, -0.75);
\draw [white](-1, 0)-- node[midway, sloped, above,black]{$\mathcal{B}_0$}(1, 0);
\draw (0,0) to [out=-5, in=210] (3.8, 0.5);
\draw (0,0) to [out=-40, in=215] (4.7, 0.3);
\draw [white](0, -5)-- node[midway, sloped, below,black]{$u=u_0$}(5, 0);
\draw [white](0, -0.65)-- node[midway, sloped, below,black]{$\mathcal{A}$}(2.8, -0.65);
\draw [white](0, -1.5)-- node[midway, sloped, below,black]{$v=v_1$}(2.2, -3.7);
\draw [thick](0, -3)--(1, -4);

\draw [thick] (0, -5)--(0,0);
\draw [thick] (5, 0)--(0,-5);
\draw[fill] (0,0) circle [radius=0.08];
\end{tikzpicture}
\end{minipage}
\begin{minipage}[!t]{0.6\textwidth}
\end{minipage}
\hspace{0.05\textwidth}
\end{center}

\noindent Here $\Gamma$ is the center (invariant under $SO(3)$). $\mathcal{B}_0$ is the first singular point along $\Gamma$.
{\color{black}$\mathcal{A}$ stands for an apparent horizon. Under spherical symmetry,
$\mathcal{A}=\{(u,v)| \mbox{ where }\partial_v r(u,v)=0\}$.} The spacetime region {\color{black}between $\mathcal{A}$ and $\mathcal{B}$} is called the {\color{black}trapped} region $\mathcal{T}$, {\color{black}where $\partial_v r(u,v)<0$ {\color{black}and $r(u,v)>0$}}. {\color{black}The} hypersurface $\mathcal{B}$ is the future boundary of this spacetime ; it is singular. In \cite{Chr.1}, Christodoulou also proved that at any point $(u,v)$ of the singular boundary $\mathcal{B}$, we have $r(u,v)=0$. \\ 

A natural question to ask is: \textit{how singular are the curvatures at this future boundary $\mathcal{B}$}{\color{black}?} In \cite{Chr.1}, Christodoulou showed that $\mathcal{B}$ is spacelike.\footnote{See Theorem 5.1 (j) in \cite{Chr.1}.} And at {\color{black} any point $(u,v)\in\mathcal{T}$ and $(u,v)$ is close to $\mathcal{B}$}, a lower bound of Kretschmann scalar holds: \footnote{See Theorem 5.1 (l) in \cite{Chr.1}.}
$$R^{\alpha\beta\gamma\delta}R_{\alpha\beta\gamma\delta}(u,v)\gtrsim \f{1}{r(u,v)^6}.$$
To prove the lower bounds, Christodoulou used an ODE type estimates: in \cite{Chr.1}, by algebraic calculations, it can be showed that at $(u,v)$
\begin{equation}\label{mass inequality}
R_{\a\b\gamma\delta}R^{\a\b\gamma\delta}(u,v)\geq \f{32\, m(u,v)^2}{r(u,v)^6},
\end{equation}
where $m(u,v)$ is the Hawking mass of $S_{u,v}$ defined in (\ref{Hawking mass}). Remarkably, $m(u,v)$ satisfies an ODE type monotone property: in the trapped region $\mathcal{T}$, it holds that $\partial_u m(u,v)\geq 0$.

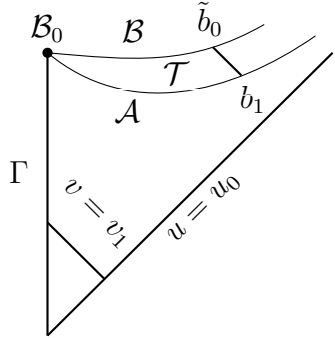
\begin{center}
\begin{minipage}[!t]{0.4\textwidth}
\begin{tikzpicture}[scale=0.75]
\draw [white](-1, -2.5)-- node[midway, sloped, above,black]{$\Gamma$}(0, -2.5);
\draw [white](0, 0)-- node[midway, sloped, above,black]{$\mathcal{B}$}(3, 0);
\draw [white](2.5, 0.1)-- node[midway, sloped, above,black]{$\t{b}_0$}(3.2, 0.1);
\draw [white](3, -0.4)-- node[midway, sloped, below,black]{$b_1$}(4.2, -0.4);
\draw [white](0, -0.75)-- node[midway, sloped, above,black]{$\mathcal{T}$}(4.5, -0.75);
\draw [white](-1, 0)-- node[midway, sloped, above,black]{$\mathcal{B}_0$}(1, 0);
\draw (0,0) to [out=-5, in=210] (3.8, 0.5);
\draw (0,0) to [out=-40, in=215] (4.7, 0.3);
\draw [white](0, -5)-- node[midway, sloped, below,black]{$u=u_0$}(5, 0);
\draw [white](0, -0.65)-- node[midway, sloped, below,black]{$\mathcal{A}$}(2.8, -0.65);
\draw [white](0, -1.5)-- node[midway, sloped, below,black]{$v=v_1$}(2.2, -3.7);
\draw [thick](0, -3)--(1, -4);

\draw [thick] (0, -5)--(0,0);
\draw [thick] (5, 0)--(0,-5);
\draw [thick] (2.9, 0.1)--(3.4, -0.4);
\draw[fill] (0,0) circle [radius=0.08];
\end{tikzpicture}
\end{minipage}
\begin{minipage}[!t]{0.6\textwidth}
\end{minipage}
\hspace{0.05\textwidth}
\end{center}
Fix $\t{b}_0\in {\color{black}\mathcal{T} \mbox{ and } \t{b}_0 \mbox{ close to } \mathcal{B}}$. Assume $\t{b}_0$ has coordinate $(\t{u}_0, \t{v}_0)$ and $b_1\in \mathcal{A}$ has coordinate $(\t{u}_1, \t{v}_0)$. Then at $\t{b}_0$ we have
$$R_{\a\b\gamma\delta}R^{\a\b\gamma\delta}(\t{u}_0,\t{v}_0)\geq \f{32\, m(\t{u}_0,\t{v}_0)^2}{r(\t{u}_0, \t{v}_0)^6}\geq \f{32\, m(\t{u}_1,\t{v}_0)^2}{r(\t{u}_0, \t{v}_0)^6}= \f{8\, r(\t{u}_1,\t{v}_0)^2}{r(\t{u}_0, \t{v}_0)^6}.$$
For the second inequality, we use $\partial_u m(u,v)\geq 0$. And for the last {\color{black}identity}, we use that along apparent horizon $\mathcal{A}$ it holds that $\partial_v r(\t{u}_1, \t{v}_0)=0$ and thus $m(\t{u}_1,\t{v}_0)=r(\t{u}_1,\t{v}_0)/2$.  Hence we get at {\color{black}$\t{b}_0\in \mathcal{T}$ near $\mathcal{B}$}
$$R_{\a\b\gamma\delta}R^{\a\b\gamma\delta}(\t{u}_0,\t{v}_0) \gtrsim \f{1}{r(\t{u}_0, \t{v}_0)^6}{\color{black}.}$$
{\color{black}This derives} the lower bounds of $R_{\a\b\gamma\delta}R^{\a\b\gamma\delta}$ {\color{black}close to} $\mathcal{B}$.\\
}

\noindent How about the upper bound? Following the \underline{qualitative} \textit{extension principle} \footnote{It states that for characteristic initial data prescribed on initial incoming and outgoing hypersurfaces $\{(u,v)\}$, where $r(u,v)\geq \e>0$, then for (\ref{ES}) the local existence towards the future can be proved.} established by Christodoulou in \cite{Chr.1.5}, it can be proved that at any point {\color{black}$(u,v)\in\mathcal{T}$ near $\mathcal{B}$}
$$R^{\alpha\beta\gamma\delta}R_{\alpha\beta\gamma\delta}(u,v)\lesssim \exp\big(\exp(\f{1}{r(u,v)})\big).$$
To get a better upper bound, we need to give a different proof and we need to improve all the estimates into \underline{quantitive sharp} estimates. In this article, we improve the double-exponential upper bounds to polynomial rates.

\begin{theorem}\label{thm1.1}
{\color{black}With the same characteristic initial data Christodoulou used in \cite{Chr.1}}, for the dynamical spacetime solutions of (\ref{ES}) under spherical symmetry, inside a {\color{black}trapped} region, at any point {\color{black}$(u,v)\in \mathcal{T}$ and $(u,v)$ is close to $\mathcal{B}$}, there exists a positive number $N$ (depending on the initial data at an earlier time), such that
$$R^{\alpha\beta\gamma\delta}R_{\alpha\beta\gamma\delta}(u,v)\lesssim \f{1}{r(u,v)^N}.$$
\end{theorem}
\begin{remark}
With the previously mentioned lower bound, we have
$$\f{1}{r(u,v)^6} \lesssim R^{\alpha\beta\gamma\delta}R_{\alpha\beta\gamma\delta}(u,v)\lesssim \f{1}{r(u,v)^N}.$$
Hence polynomial blow-up upper bounds are \underline{sharp}. 
\end{remark}

{\color{black}
\begin{remark}
By using (\ref{mass inequality}), we also bound {\color{black}the} Hawking mass. And it holds
$$m(u,v)\lesssim \f{1}{r(u,v)^{\f{N}{2}-3}},$$
where $N\geq 6$ is {\color{black}a} constant depends on initial data.  

\end{remark}
} 

\noindent To obtain Theorem \ref{thm1.1}, a crucial step is to prove sharp blow-up rates for $\partial_u \phi$ and $\partial_v \phi$.
Here we have
\begin{theorem}\label{prop1.2}
Under the same assumptions as in Theorem \ref{thm1.1}, at any point {\color{black}$(u,v)\in \mathcal{T}$ and $(u,v)$ is close to $\mathcal{B}$}, there {\color{black}exist} positive {\color{black}numbers} $D_1$ and $D_2$ (depending on the initial data), such that
$$|\partial_u \phi(u,v)|\leq \f{D_1}{r(u,v)^2}, \quad \quad |\partial_v \phi(u,v)|\leq \f{D_2}{r(u,v)^2}.$$
\end{theorem}
\begin{remark}
Here the exponent $2$ is sharp.
\end{remark}
\begin{remark}
These estimates further imply $|\phi|\lesssim |\log r|$. And the $N$ in Theorem \ref{thm1.1} depends on the values of $D_1$ and $D_2$. 
\end{remark}

{\color{black}
\noindent As a case study, in Section \ref{A Case Study} we also provide more precise upper bounds for spacetimes close to Schwarzschild metric:\\
 
\noindent {\textbf Theorem \ref{close to Schwarzschild}.} We consider the trapezoid region $T_0$ below. 

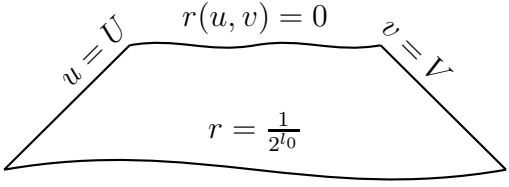
\begin{minipage}[!t]{0.55\textwidth}
\begin{tikzpicture}[scale=0.55]
\draw [white](-3, 0)-- node[midway, sloped, above,black]{$r(u,v)=0$}(3, 0);
\draw [white](-4, -1)-- node[midway, sloped, above,black]{$u=U$}(-3, 0);
\draw [white](3, 0)-- node[midway, sloped, above,black]{$v=V$}(4, -1);
\draw [thick] (-3,0) to [out=10,in=-170] (0,0);
\draw [thick] (0,0) to [out=10,in=-170] (3,0);
\draw [thick] (-6,-3)--(-3,0);
\draw [thick] (3,0)--(6,-3);
\draw [white](-5, -2.8)-- node[midway, sloped, above,black]{$r=\f{1}{2^{l_0}}$}(5, -2.8);
\draw [thick] (-6, -3) to [out=10,in=-170] (6,-3);
\end{tikzpicture}
\end{minipage}
\begin{minipage}[!t]{0.4\textwidth}
For $l_0$ being a large positive constant, we prescribe initial data along $r=1/2^{l_0}$: requiring
\end{minipage}
\hspace{0.05\textwidth}

$$|\partial_v r+\f{M}{r}|\leq o_0(1)\cdot \f{M}{r}, \quad \quad |\partial_u r+\f{M}{r}|\leq o_0(1)\cdot \f{M}{r},$$
$$|\O^2-\f{2M}{r}|\leq o_0(1)\cdot \f{M}{r},$$
$$|\partial_u \phi|\leq o_0(1)\cdot \f{1}{r^2}, \quad \quad |\partial_v \phi|\leq o_0(1)\cdot \f{1}{r^2},$$
where $o_0(1)$ is a small positive number depending on initial data. Then for the dynamical spacetime solutions of (\ref{ES}) under spherical symmetry, under the prescribed initial data, in the {\color{black}open trapezoid} region above, we have
\begin{equation}
|R^{\alpha\beta\rho\sigma}R_{\alpha\beta\rho\sigma}|\lesssim \f{1}{r^{6+o_{{\color{black}0}}(1)^2}}.
\end{equation}
\begin{remark}
From the above theorem, we can also conclude that as the initial perturbation $o_{{\color{black}0}}(1)\rightarrow 0$, the upper bound of blow-up rate $6+o_{{\color{black}0}}(1)^2\rightarrow 6$.  
\end{remark}

}

{\color{black}
\subsection{New Ingredients.} 
\begin{enumerate}
\item In this paper we study the blow-up mechanism, which is NOT ODE type. And we find an interesting log structure.\\

\noindent To derive the blow-up upper bounds, we use the full expression of $R_{\a\b\gamma\delta}R^{\a\b\gamma\delta}$. See (\ref{Kretschmann2}). And the sharp upper bounds of $\O^{-2}(u,v)$ are crucial. To bound $\O^{-2}(u,v)$, we need to use a wave-type equation for $\log\O(u,v)$:
\begin{equation}\label{uv Omega 1}
r^2\partial_u \partial_v \log \O=\partial_u r \partial_v r+\f14\O^2-r^2 \partial_u \phi \partial_v \phi.
\end{equation}
\noindent \textit{The $\log$ structure here will play a very important role.}{\color{black}\footnote{More details of this $\log$ structure will be explained in next page.}}
From the above equation, we also see that to bound $\O^{-2}$ the sharp bounds for $\partial_u \phi, \partial_v \phi$ are also required. This requires a thorough analysis of the following wave equation as well
\begin{equation*}
r\partial_u \partial_v \phi=-\partial_u r \partial_v \phi-\partial_v r \partial_u \phi.
\end{equation*}

\noindent In this paper, we explore the log structure and study the above wave equations. We derive the sharp upper bounds for $\partial_u \phi, \partial_v \phi$ and $\O^{-2}$. For Einstein-scalar field system, these bounds are new.   \\

\item Our blow-up upper bounds are optimal. {\color{black}The $\log$ structure in \eqref{uv Omega 1} is crucially used.} \\

\noindent We proceed to derive the bounds for $\O^2(u,v)$. Unlike $\O^{-2}(u,v)$, via a monotonic property (see Section \ref{upper bounds for Omega}), we can prove $\O^2(u,v)\lesssim 1/r(u,v)$. The lower bound of $\O^2(u,v)$, that is the upper bound of $\O^{-2}(u,v)$ is much harder and it is the key for the polynomial blow-up upper bounds.\\

\noindent In \cite{Chr.1}, Christodoulou showed that at each $b_0\in \mathcal{B}$, it holds that $|r\partial_u r|$ and $|r\partial_v r|$ are bounded and are close to some non-zero constant depending on $b_0$.  Let's first pretend to ignore $-r^2\partial_u \phi \partial_v \phi$ term in (\ref{uv Omega 1}). From 
$$|\partial_u r\partial_v r+\f14 \O^2|\lesssim \f{1}{r^2}+\f{1}{r}\lesssim \f{1}{r^2}, $$
we have 
$$|r^2\partial_u \partial_v \log \O(u,v)| \leq {1}/{r(u,v)^2}, \mbox{i.e. } |\partial_u \partial_v \log \O(u,v)| \leq {1}/{r(u,v)^4}.$$  
With the fact $\{r\partial_u r, r\partial_v r\}$ are close to non-zero constants, last inequality above implies 
\begin{equation*}
\begin{split}
&|\partial_v \log \O(u,v)|\\
\lesssim& \mbox{initial data}+\int_{u_1}^{u}\f{1}{r(u',v)^4}du'\lesssim \mbox{initial data}+|\int_{u_1}^{u}\f{\partial_u r}{r(u',v)^3}du'|\\
\lesssim& \mbox{initial data}+\int_{r(u_1,v)}^{r(u,v)}\f{1}{r(u',v)^3}dr\lesssim \mbox{initial data}+\f{1}{r(u,v)^2}. \, \mbox{ And}
\end{split}
\end{equation*}
\begin{equation*}
\begin{split}
&|\log \O(u,v)|\\
\lesssim& \mbox{initial data}+\int_{v_1}^{v}\f{1}{r(u,v')^2}dv'\lesssim \mbox{initial data}+|\int_{v_1}^{v}\f{\partial_v r}{r(u,v')}dv'|\\
\lesssim& \mbox{initial data}+\int_{r(u,v_1)}^{r(u,v)}\f{1}{r(u,v')}dr\lesssim \mbox{initial data}+\log r(u,v). \\
&\mbox{Thus  }  \log \f{1}{\O(u,v)^2}\lesssim \mbox{initial data}+\log \f{1}{r(u,v)}. \mbox{ Using the $\log$}
\end{split}
\end{equation*}
\noindent structure, this means that there exists a positive constant $C$ such that 
$$|\O(u,v)^{-2}|\lesssim r(u,v)^{C}.$$ 
And $C$ depends on the constants in above inequalities. \\

\noindent Now we take the term  $-r^2\partial_u \phi \partial_v \phi$ into account. Our goal is to show 
$$|\partial_u \phi(u,v)|\lesssim \f{1}{r(u,v)^2}, \quad |\partial_v \phi(u,v)|\lesssim \f{1}{r(u,v)^2}.$$
Then we would have $|-r^2\partial_u \phi \partial_v \phi|\lesssim \f{1}{r^2}$, and it would be the same blow-up rates to $|\partial_u r\partial_v r+\f14 \O^2|\lesssim \f{1}{r^2}$. Repeat the calculation above, $|\O(u,v)^{-2}|\leq r(u,v)^{C}$ still holds for some positive constant $C$.  \\

\noindent On the other hand, if we cannot obtain the optimal exponent $2$. For $0<\epsilon\ll1$, assume that we could only prove 
$$|\partial_u \phi(u,v)|\lesssim \f{1}{r(u,v)^{2+\epsilon}}, \quad |\partial_v \phi(u,v)|\lesssim \f{1}{r(u,v)^{2+\epsilon}}.$$
For this case, we get 
$$|\partial_u \partial_v \log \O(u,v)| \leq {1}/{r(u,v)^{4+2\e}}.$$  
And it implies 
\begin{equation*}
\begin{split}
&|\partial_v \log \O(u,v)|\\
\lesssim& \mbox{initial data}+\int_{u_1}^{u}\f{1}{r(u',v)^{4+2\e}}du'\lesssim \mbox{initial data}+|\int_{u_1}^{u}\f{\partial_u r}{r(u',v)^{3+2\e}}du'|\\
\lesssim& \mbox{initial data}+\int_{r(u_1,v)}^{r(u,v)}\f{1}{r(u',v)^{3+2\e}}dr\lesssim \mbox{initial data}+\f{1}{r(u,v)^{2+2\e}}. \, \mbox{And}
\end{split}
\end{equation*}

\begin{equation*}
\begin{split}
&|\log \O(u,v)|\\
\lesssim& \mbox{initial data}+\int_{v_1}^{v}\f{1}{r(u,v')^{2+2\e}}dv'\lesssim \mbox{initial data}+|\int_{v_1}^{v}\f{\partial_v r}{r(u,v')^{1+2\e}}dv'|\\
\lesssim& \mbox{initial data}+\int_{r(u,v_1)}^{r(u,v)}\f{1}{r(u,v')^{1+2\e}}dr\lesssim \mbox{initial data}+\f{1}{r(u,v)^{2\e}}. \mbox{ Thus,}\\
\end{split}
\end{equation*}
$$\mbox{we have }  \log \f{1}{\O(u,v)^2}\lesssim \mbox{initial data}+\f{1}{r(u,v)^{2\e}} \mbox{ and } \O(u,v)^{-2}\lesssim e^{\f{1}{r(u,v)^{2\e}}}.$$

\noindent By the expression of Kretschmann scalar (\ref{Kretschmann2}), this would lead to an upper bound \underline{exponential }of $1/r$, \underline{not polynomial.}  \\

\noindent Hence, the key is to obtain the sharp upper bounds 
\begin{equation}\label{goal of phi}
|\partial_u \phi(u,v)|\lesssim \f{1}{r(u,v)^2}, \quad |\partial_v \phi(u,v)|\lesssim \f{1}{r(u,v)^{2}}.
\end{equation}
And we achieve this goal in our Theorem \ref{prop1.2}. The proof is based on a crucial improved estimates for $r\partial_u r, r\partial_v r$. These improved estimates may have other applications. And we highlight them in next paragraph.\\

\item In this paper, we found crucial geometric improved estimates for $r\partial_u r(u,v)$ and $r\partial_v r(u,v)$. Our result is general and non-perturbative. We do not require our spacetimes to be close to Schwarzschild metric.\\

\noindent To prove (\ref{goal of phi}), we take two steps. For the first step, we employ an important observation by Christodoulou and we reprove it with double null foliation in Proposition \ref{Prop 4.1}:\\

\begin{minipage}[!t]{0.3\textwidth}
\begin{tikzpicture}[scale=0.7]
\draw [white](-1, -2.5)-- node[midway, sloped, above,black]{$\Gamma$}(0, -2.5);
\draw [white](0, 0)-- node[midway, sloped, above,black]{$\mathcal{B}$}(3, 0);
\draw [white](2.5, 0.1)-- node[midway, sloped, above,black]{$b_0$}(3.2, 0.1);
\draw [white](3, -0.4)-- node[midway, sloped, below,black]{$b_1$}(4.2, -0.4);
\draw [white](0, -0.75)-- node[midway, sloped, above,black]{$\mathcal{T}$}(4.5, -0.75);
\draw [white](-1, 0)-- node[midway, sloped, above,black]{$\mathcal{B}_0$}(1, 0);
\draw (0,0) to [out=-5, in=210] (3.8, 0.5);
\draw (0,0) to [out=-40, in=215] (4.7, 0.3);
\draw [white](0, -5)-- node[midway, sloped, below,black]{$u=u_0$}(5, 0);
\draw [white](0, -0.65)-- node[midway, sloped, below,black]{$\mathcal{A}$}(2.8, -0.65);
\draw [white](0, -1.5)-- node[midway, sloped, below,black]{$v=v_1$}(2.2, -3.7);
\draw [thick](0, -3)--(1, -4);

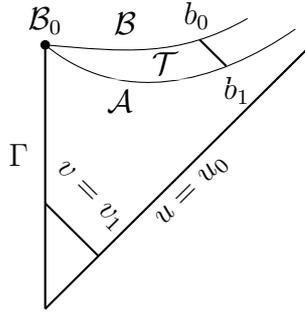
\draw [thick] (0, -5)--(0,0);
\draw [thick] (5, 0)--(0,-5);
\draw [thick] (2.9, 0.1)--(3.4, -0.4);
\draw[fill] (0,0) circle [radius=0.08];
\end{tikzpicture}
\end{minipage}
\begin{minipage}[!t]{0.58\textwidth}
\noindent Given the same characteristic initial value problem for (\ref{ES}) as above. Assume $b_0\in \mathcal{B}$ and $b_0$ has coordinate $(\t{u}_0, \t{v}_0)$, then as $(\t{u}, \t{v}_0)\rightarrow (\t{u}_0, \t{v}_0)$ we have 
$-(r\partial_v r)(\t{u}, \t{v}_0)\rightarrow E(\t{v}_0) \mbox{ as } u\rightarrow \t{u}_0-, $ where $E$ is a positive continuous function. Similarly,  as $(\t{u}_0, \t{v})\rightarrow (\t{u}_0, \t{v}_0)$ we have 
$-(r\partial_u r)(\t{u}_0, \t{v})\rightarrow E^*(\t{u}_0), \mbox{ as } v\rightarrow \t{v}_0-,$ where $E^*$ is a positive continuous function.
\end{minipage}
\hspace{0.07\textwidth}

\noindent {\color{black}Proposition \ref{Prop 4.1}} shows that near $b_0\in \mathcal{B}$, there exists positive constants $C_1$ and $C_2$, and by the continuity of $E(v)$ and $E^*(u)$, for points close to $b_0$, the followings hold
$$|r\partial_u r+C_1|=o(1), \quad |r\partial_v r+C_2|=o(1).$$
With these and an energy estimate, we first obtain Proposition \ref{preliminary}: for $0<\a\ll 1$, in the region of interest, we have
$$|\partial_u \phi(u,v)|\lesssim \f{1}{r(u,v)^{3+\a}}, \quad |\partial_v \phi(u,v)|\lesssim \f{1}{r(u,v)^{3+\a}}.$$

\noindent The next is one of the key points in this paper. In Proposition \ref{improved ru rv}, together with a \underline{novel} \underline{geometric} argument and by applying bounds in Proposition \ref{preliminary} for $\partial_u \phi, \partial_v \phi$, we have a crucial \underline{quantitive} improvement of the estimates for $r\partial_u r$ and $r\partial_v r$: we obtain that for any $(u,v)$ close to $b_0$, it holds
$$|r\partial_u r+C_1|(u,v)\leq 2r(u,v)^{\f{1}{100}}, \quad |r\partial_v r+C_1|(u,v)\leq 2r(u,v)^{\f{1}{100}}.$$
This crucial improvement enables us to correct a potential divergent $\log r(u,v)$ term with a finite constant. (See the proof in Theorem \ref{prop1.2}.)

The conclusion and the argument in Proposition \ref{improved ru rv} will lead to future applications. With these crucial improvements, in Section \ref{section sharp phi} via using a constant $r(u,v)$ foliation, we prove Theorem \ref{prop1.2}:
$$|\partial_u \phi(u,v)|\lesssim \f{1}{r(u,v)^2}, \quad |\partial_v \phi(u,v)|\lesssim \f{1}{r(u,v)^{2}}.$$
Note that our proof in Proposition \ref{improved ru rv} is very general. We don't need our spacetimes to be close to Schwarzschild metric. Hence, the blow-up upper bounds we derived are also general. And our proof is not perturbative.

\end{enumerate}
}

\subsection{Background}

In a series of celebrated papers \cite{Chr.1}-\cite{Chr.3}, Christodoulou proved \textit{weak cosmic censorship} for (\ref{ES}) under spherical symmetry. He showed that for generic initial data, the singularities formed in the evolution of (\ref{ES}) are hidden inside black hole regions.

One could further ask: inside black holes, what are the future boundaries like? This question is related to \textit{strong cosmic censorship}. For spacetimes like Kerr and Reissner-Nordstr\"om black holes, their future boundaries are null hypersurfaces, called Cauchy horizons. In recent breakthrough papers \cite{LO3} by Luk and Oh and \cite{DL} by Dafermos and Luk, the regularities of Cauchy horizons are studied in {\color{black}detail}.  Interested readers are also {\color{black}referred} to \cite{CGNS}-\cite{Fra}, \cite{Gaj2}-\cite{PI} and \cite{VDM, VDM2}.

For spacetimes close to a Schwarzschild black hole, their future boundaries could be more singular {\color{black}than the spacetimes near Kerr or Reissner-Nordstr\"om black holes}. In \cite{Sbi} Sbierski proved the $C^0$-inextendibility of Schwarzschild spacetime. In \cite{Fou} Fournodavlos studied the backward stability of the Schwarzschild singularity for Einstein vacuum equations; Alexakis and {\color{black}Fournodavlos} \cite{AF} are exploring the forward stability problem under axial symmetry. 

Fournodavlos and Sbierski \cite{FS} also studied the asymptotic behaviours of {\color{black}linear waves} in the interior region of Schwarzschild spacetime. {\color{black}For linear wave equation in Schwarzschild background
$$\Box_{g_{Sch}}\phi=0,$$
close to spacelike singularity $r(u,v)=0$, they proved that $|\phi(u,v)|\lesssim |\log r(u,v)|$ and also gave the leading order asymptotic behaviours. Note that their bounds are consistent with the upper bounds we derive in Theorem \ref{prop1.2}. Their analysis is for linear wave equation in precise Schwarzschild background and they don't impose symmetry assumption. For our results, we impose spherical symmetry, but our theorem is for the full Einstein-scalar field system and our spacetime metric could be far away from Schwarzschild metric.} \\

The future boundary $\mathcal{B}$ in \cite{Chr.1} and Schwarzschild singularities share some common properties: in both spacetimes, the singular boundaries are spacelike. And for any point $(u,v)$ along $\mathcal{B}$, we have $r(u,v)=0.$ But spacetimes in \cite{Chr.1} are much more general. The future boundaries $\mathcal{B}$ in \cite{Chr.1} are beyond {\color{black}the} perturbative {\color{black}regimes} of Schwarzschild singularities. In this following, we will explore how singular $\mathcal{B}$ could be.

\section{Acknowlegements}
The authors would like to thank an anonymous referee for comments on a previous version. X.A. want to thank Spyros
Alexakis, Sung-Jin Oh and Willie Wong for stimulating discussions.

\section{Settings and Basic Control of Geometric Quantities}
Under spherical symmetry, with double null foliations, we have the following ansatz for metric of the $3+1$-dimensional spacetime:
$$g_{\mu\nu}dx^{\mu}dx^{\nu}=-\O^2(u,v)dudv+r^2(u,v)\big(d\theta^2+\sin^2\theta d\phi^2\big).$$
With this ansatz, {\color{black}the} Einstein scalar field system
$$\mbox{Ric}_{\mu\nu}-\f12Rg_{\mu\nu}=2T_{\mu\nu},$$
$$T_{\mu\nu}=\partial_{\mu}\phi \partial_{\nu}\phi-\f12g_{\mu\nu}\partial^{\sigma}\phi \partial_{\sigma}\phi,$$
{\color{black}can} be rewritten as
\begin{equation}\label{eqn r}
r\partial_{u}\partial_v r=-\partial_u r \partial_v r-\f14\O^2,
\end{equation}

\begin{equation}\label{eqn Omega}
r^2\partial_u \partial_v \log \O=\partial_u r \partial_v r+\f14\O^2-r^2 \partial_u \phi \partial_v \phi,
\end{equation}

\begin{equation}\label{eqn phi}
r\partial_u \partial_v \phi=-\partial_u r \partial_v \phi-\partial_v r \partial_u \phi,
\end{equation}

\begin{equation}\label{eqn u}
\partial_u (\O^{-2}\partial_u r)=-r\O^{-2}(\partial_u \phi)^2,
\end{equation}

\begin{equation}\label{eqn v}
\partial_v(\O^{-2}\partial_v r)=-r\O^{-2}(\partial_v \phi)^2.
\end{equation}

For later use, we define Hawking mass $m(u,v)$ for a two-sphere $S_{u,v}$ implicitly by
$$1-\f{2m}{r}=-4\Omega^{-2}\partial_u r \partial_v r.$$
We further introduce the dimensionless quantity
$$\mu=\f{2m}{r}.$$
Note that along the apparent horizon $\mathcal{A}$, we have $\partial_v r=0$. This implies
$$1-\f{2m}{r}=0, \quad \mu=1, \quad 2m=r, \quad \mbox{along} \quad \mathcal{A}.$$
And inside the trapped region $\mathcal{T}$, we have $\partial_v r<0, \partial_u r<0$. It follows that
$$1-\f{2m}{r}<0, \quad \mu>1, \quad 2m>r, \quad \mbox{in} \quad \mathcal{T}.\\$$

With $m$ and $\mu$, we could rewrite (\ref{eqn r})-(\ref{eqn v}) and further have
\begin{equation}\label{r v}
\partial_u (\partial_v r)=\f{\mu}{(1-\mu)r}\partial_v r \partial_u r,
\end{equation}

\begin{equation}\label{r u}
\partial_v (\partial_u r)=\f{\mu}{(1-\mu)r}\partial_v r \partial_u r,
\end{equation}

\begin{equation}\label{m u}
2\partial_u r \partial_u m=(1-\mu)r^2 (\partial_u \phi)^2,
\end{equation}

\begin{equation}\label{m v}
2\partial_v r \partial_v m=(1-\mu)r^2 (\partial_v \phi)^2.\\
\end{equation}

\subsection{Estimates for $\partial_u r$ and $\partial_v r$} Let us recall a proposition by Christodoulou in \cite{Chr.1.5} and give a proof in double null foliation:
\begin{proposition}\label{Prop 4.1} (Proposition 8.2 in \cite{Chr.1.5}):
{\color{black}Given the same characteristic initial data in a double-null foliation as above. Assume $b_0\in \mathcal{B}$ has coordinate $(u^*(v), v)$, we have 
$$-(r\partial_v r)(u,v)\rightarrow E(v)$$
as $u\rightarrow u^*(v)-$, where $E$ is a positive continuous function of $v$. Similarly, assume each $b_0\in \mathcal{B}$ has coordinate $(u, v^*(u))$. Then it also holds
$$-(r\partial_u r)(u,v)\rightarrow E^{*}(u)$$
as $v\rightarrow v^*(u)$, where $E^*$ is a positive continuous function of $u$.} Also, the non-central component $\mathcal{B}\backslash\mathcal{B}_0$ of the singular boundary $\mathcal{B}$ is a $C^1$ strictly spacelike curve, i.e., the functions $v^*$ and $u^*$ are strictly decreasing $C^1$ functions.
\end{proposition}

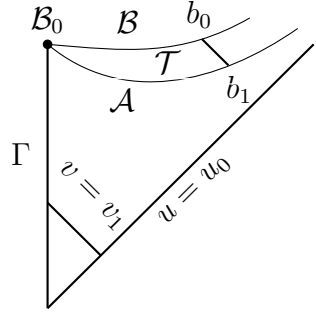
\begin{center}
\begin{tikzpicture}[scale=0.7]
\draw [white](-1, -2.5)-- node[midway, sloped, above,black]{$\Gamma$}(0, -2.5);
\draw [white](0, 0)-- node[midway, sloped, above,black]{$\mathcal{B}$}(3, 0);
\draw [white](2.5, 0.1)-- node[midway, sloped, above,black]{$b_0$}(3.2, 0.1);
\draw [white](3, -0.4)-- node[midway, sloped, below,black]{$b_1$}(4.2, -0.4);
\draw [white](0, -0.75)-- node[midway, sloped, above,black]{$\mathcal{T}$}(4.5, -0.75);
\draw [white](-1, 0)-- node[midway, sloped, above,black]{$\mathcal{B}_0$}(1, 0);
\draw (0,0) to [out=-5, in=210] (3.8, 0.5);
\draw (0,0) to [out=-40, in=215] (4.7, 0.3);
\draw [white](0, -5)-- node[midway, sloped, below,black]{$u=u_0$}(5, 0);
\draw [white](0, -0.65)-- node[midway, sloped, below,black]{$\mathcal{A}$}(2.8, -0.65);
\draw [white](0, -1.5)-- node[midway, sloped, below,black]{$v=v_1$}(2.2, -3.7);
\draw [thick](0, -3)--(1, -4);

\draw [thick] (0, -5)--(0,0);
\draw [thick] (5, 0)--(0,-5);
\draw [thick] (2.9, 0.1)--(3.4, -0.4);
\draw[fill] (0,0) circle [radius=0.08];
\end{tikzpicture}
\hspace{0.07\textwidth}
\end{center}

\begin{proof}
In the trapped region from (\ref{m u}), we have $\partial_u m{\color{black}\geq}0$. For each $v$, we assume the 2-sphere $b_1=(u_A(v),v)$ lays on $\mathcal{A}$ with radius $r_A(v)$. Hence for any $(u,v)\in \mathcal{T}$, we have
$$2m(u,v)\geq 2m(u_A(v),v)=r_A(v).$$
Thus in $\mathcal{T}$, we have
$$\mu(u,v)\geq \f{r_A(v)}{r(u,v)}.$$
From (\ref{r v}), we obtain
$$\f{1}{|\partial_v r|}\cdot\f{\partial |\partial_v r|}{\partial u}=-\f{\mu}{\mu-1}\cdot \f{\partial \log r}{\partial u},$$
and hence
$$\f{-\partial \log |\partial_v r|/\partial u}{\partial \log r/\partial u}=\f{\mu}{\mu-1}.$$
We conclude that in $\mathcal{T}$
$$1<\f{-\partial \log |\partial_v r|/\partial u}{\partial \log r/\partial u}\leq \f{r_A}{r_A-r}.$$
From $$1<\f{-\partial \log |\partial_v r|/\partial u}{\partial \log r/\partial u},$$
we have
$$0<\f{\partial \log r|\partial_v r|}{\partial u}.$$
Taking $v$ fixed and integrating with respect to $u$ we obtain
$$r|\partial_v r|(u,v)>r|\partial_v r|(u_1,v),$$
{\color{black}where} we require $u_A(v)<u_1\leq u \leq u^*(v)$, and $(u^*(v), v)$ is according to the singular boundary along constant $v$.

Using
$$\f{-\partial \log |\partial_v r|/\partial u}{\partial \log r/\partial u}\leq \f{r_A}{r_A-r},$$
we have
$$\f{\partial \log r|\partial_v r|}{\partial u}\leq -\f{\partial \log r}{\partial u}\cdot \f{r}{r_A-r}=-\f{\partial r}{\partial u}\cdot \f{1}{r_A-r}=\f{\partial \log (r_A-r)}{\partial u}.$$
This implies
$$\log \f{r|\partial_v r|(u,v)}{r|\partial_v r|(u_1, v)}\leq \log \f{r_A(v)-r(u,v)}{r_A(v)-r(u_1,v)},$$
and hence
$$\f{r|\partial_v r|(u,v)}{r|\partial_v r|(u_1, v)}\leq \f{r_A(v)-r(u,v)}{r_A(v)-r(u_1,v)}\leq \f{r_A(v)}{r_A(v)-r(u_1,v)}.$$
Therefore, we conclude
$$\f{r_A(v)}{r_A(v)-r(u_1,v)} (r\partial_v r)(u_1, v)\leq (r\partial_v r)(u,v)< (r\partial_v r)(u_1, v).$$
A direct checking gives
$$\limsup_{u\rightarrow u^*(v)-}\big(-(r\partial_v r)(u,v)\big)-\liminf_{u\rightarrow u^*(v)-}\big(-(r\partial_v r)(u,v)\big)\leq \f{r(u_1,v)}{r_A-r(u_1,v)}\big(-(r\partial_v r)(u_1,v)\big).$$
Taking $u_1\rightarrow u^*(v)-$, we have that $r(u_1, v)\rightarrow 0$. Therefore, we {\color{black}conclude}
$$-(r\partial_v r)(u,v) \, \, \mbox{tends to a positive limit} \, \, E(v) \, \, \mbox{as} \, \, u\rightarrow u^*(v)-.$$
Similar arguments work for $-(r\partial_ u{\color{black}r})(u,v)$. And we have
$$-(r\partial_u r)(u,v) \, \, \mbox{tends to a positive limit} \, \, E^*(u) \, \, \mbox{as} \, \, v\rightarrow v^*(u)-.$$
Last, since
$$\f{\partial r^2}{\partial v}=2r \partial_v r, \quad  \quad \quad \f{\partial r^2}{\partial u}=2r \partial_u r,$$
the above yields that function $r^2(u,v)$ extends to a $C^1$ function to $\mathcal{Q}\cup(\mathcal{B}\backslash \mathcal{B}_0)$. And thus $\mathcal{B}\backslash\mathcal{B}_0$ is a $C^1$ curve.

\end{proof}

\subsection{Upper Bounds for $\Omega^2$}\label{upper bounds for Omega}
From (\ref{eqn u}) we have
$$\partial_u (\Omega^{-2}\cdot -\partial_u r)=r\O^{-2}(\partial_u \phi)^2\geq 0.$$
Thus, for $(u,v)\in \mathcal{T}$ and $(u_0(v),v)\in \mathcal{A}$ we have
$$\f{-\partial_u r (u,v)}{\Omega^2(u,v)}\geq \f{-\partial_u r (u_0(v),v)}{\Omega^2(u_0(v),v)}=c_0, \,\, \mbox{for some positive constant} \,\, c_0.$$
Hence
\begin{equation}\label{Omega upper}
\O^2(u,v)\leq c_0^{{\color{black}-1}}\cdot [-\partial_u r (u,v)]\leq \f{D}{r(u,v)},
\end{equation}
where $D$ is a uniform number depending on initial data.

\section{Preliminary Bounds for $\partial_u \phi$ and $\partial_v \phi$}

The aim of this and the next two sections is to prove
\begin{equation}\label{goal}
|r^2\partial_u \phi|\leq D_1, \mbox{ and } |r^2\partial_v \phi|\leq D_2,
\end{equation}
with $D_1, D_2$ being uniform numbers depending on initial data. To achieve this goal, we first derive some preliminary estimates.

\begin{proposition}\label{preliminary}
For $0<\alpha\leq 1$, in the region of interest, we have
\begin{equation}
|\partial_u \phi(u,v)|\leq \f{I_0}{r^{3+\alpha}(u,v)}, \mbox{ and } |\partial_v \phi(u,v)|\leq \f{I_0}{r^{3+\alpha}(u,v)},
\end{equation}
where $I_0$ is a uniform number depending on initial data.
\end{proposition}
\begin{proof}
Assume the whole diamond region below is in trapped region $\mathcal{T}$.

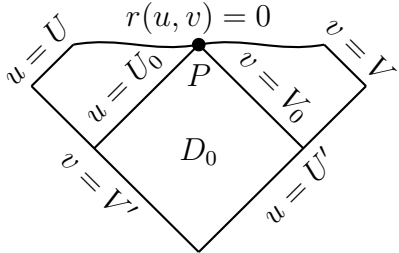
\begin{minipage}[!t]{0.4\textwidth}
\begin{tikzpicture}[scale=0.55]
\draw [white](-3, -0.2)-- node[midway, sloped, below,black]{$P$}(3, -0.2);
\draw [white](-3, -2)-- node[midway, sloped, below,black]{${\color{black}D_0}$}(3, -2);
\draw [white](-3, 0)-- node[midway, sloped, above,black]{$r(u,v)=0$}(3, 0);
\draw [white](-4, -1)-- node[midway, sloped, above,black]{$u=U$}(-3, 0);
\draw [white](3, 0)-- node[midway, sloped, above,black]{$v=V$}(4, -1);
\draw [white](0, -5)-- node[midway, sloped, below,black]{$u=U'$}(4, -1);
\draw [white](0, -5)-- node[midway, sloped, below,black]{$v=V'$}(-4, -1);
\draw [white](0, 0)-- node[midway, sloped, above, black]{$u=U_0$}(-2.8,-2.8);
\draw [white](0, 0)-- node[midway, sloped, above, black]{$v=V_0$}(2.8,-2.8);
\draw[thick] (0,0)--(-2.5,-2.5);
\draw[thick] (0,0)--(2.5, -2.5);
\draw [thick] (-3,0) to [out=10,in=-170] (0,0);
\draw [thick] (0,0) to [out=10,in=-170] (3,0);
\draw [thick] (-4,-1)--(-3,0);
\draw [thick] (3,0)--(4,-1);
\draw [thick] (-4,-1)--(0,-5);
\draw [thick] (4,-1)--(0,-5);
\draw[fill] (0,0) circle [radius=0.15];
\end{tikzpicture}
\end{minipage}
\begin{minipage}[!t]{0.5\textwidth}
We denote the rectangular region on the left to be $D_0$. Choose $D_0$ to be small enough. And we will focus on this region.
\end{minipage}
\hspace{0.05\textwidth}

\noindent From (\ref{eqn phi}), we have
$$\partial_v (r^2\partial_u \phi)^2=2r^3 \partial_v r (\partial_u \phi)^2-2r^3 \partial_u r\partial_u \phi \partial_v \phi,$$
$$\partial_u (r^2\partial_v \phi)^2=2r^3 \partial_u r (\partial_v \phi)^2-2r^3 \partial_v r\partial_u \phi \partial_v \phi.$$
For $0<\alpha\leq 1$, we have
$$\partial_v [r^{2\alpha}(r^2\partial_u \phi)^2]=(2\alpha+2)\partial_v r\cdot r^{2\alpha}\cdot r^3 (\partial_u \phi)^2-2\partial_u r\cdot r^{2\alpha}\cdot r^3 \partial_u \phi \partial_v \phi,$$
$$\partial_u [r^{2\alpha}(r^2\partial_v \phi)^2]=(2\alpha+2)\partial_u r\cdot r^{2\alpha}\cdot r^3 (\partial_v \phi)^2-2\partial_v r\cdot r^{2\alpha}\cdot r^3 \partial_u \phi \partial_v \phi.$$
By Proposition \ref{Prop 4.1}, inside a sufficiently small $D_0$ we have:
$$\partial_v r+\f{C_2}{r}=\f{o(1)}{r}, \quad \quad \partial_u r+\f{C_1}{r}=\f{o(1)}{r}.$$
We obtain
$$\partial_v [r^{2\alpha}(r^2\partial_u \phi)^2]=(2\alpha+2)\cdot[-C_2+o(1)]\cdot r^{2\alpha+2}\cdot (\partial_u \phi)^2-2\cdot[-C_1+o(1)]\cdot r^{2\alpha+2}\cdot \partial_u \phi \partial_v \phi,$$
$$\partial_u [r^{2\alpha}(r^2\partial_v \phi)^2]=(2\alpha+2)\cdot[-C_1+o(1)]\cdot r^{2\alpha+2}\cdot (\partial_v \phi)^2-2\cdot[-C_2+o(1)]\cdot r^{2\alpha+2}\cdot \partial_u \phi \partial_v \phi.$$
Adding these two expression together, we have
\begin{equation*}
\begin{split}
&\partial_v [C_2\cdot r^{2\alpha}(r^2\partial_u \phi)^2]+\partial_u [C_1\cdot r^{2\alpha}(r^2\partial_v \phi)^2]\\
=&2\alpha\cdot C_2\cdot[-C_2+o(1)]r^{2\alpha+2}(\partial_u \phi)^2+2\alpha\cdot C_1\cdot[-C_1+o(1)]r^{2\alpha+2}(\partial_v \phi)^2\\
&+2\cdot C_2\cdot[-C_2+o(1)]r^{2\alpha+2}(\partial_u \phi)^2+2\cdot C_1\cdot[-C_1+o(1)]r^{2\alpha+2}(\partial_v \phi)^2\\
&-2\cdot C_2\cdot[-C_1+o(1)]\cdot r^{2\alpha+2}\cdot \partial_u \phi \partial_v \phi-2\cdot C_1\cdot[-C_2+o(1)]\cdot r^{2\alpha+2}\cdot \partial_u \phi \partial_v \phi\\
=&2\alpha\cdot C_2\cdot[-C_2+o(1)]r^{2\alpha+2}(\partial_u \phi)^2+2\alpha\cdot C_1\cdot[-C_1+o(1)]r^{2\alpha+2}(\partial_v \phi)^2\\
&-2C_2^2\cdot r^{2\alpha+2}(\partial_u \phi)^2-2C_1^2\cdot r^{2\alpha+2}(\partial_v \phi)^2\\
&+2\cdot o(1)\cdot C_2 \cdot r^{2\alpha+2}(\partial_u \phi)^2+2\cdot o(1)\cdot C_1\cdot r^{2\alpha+2}(\partial_v \phi)^2\\
&+2C_1C_2\cdot r^{2\alpha+2}\partial_u \phi \partial_v \phi+2C_1C_2\cdot r^{2\alpha+2}\partial_u \phi \partial_v \phi\\
&-2\cdot o(1)\cdot C_1\cdot r^{2\alpha+2}\partial_u \phi \partial_v \phi-2\cdot o(1) \cdot C_2\cdot r^{2\alpha+2}\partial_u \phi \partial_v \phi\\
\leq&2\alpha\cdot C_2\cdot[-C_2+o(1)]r^{2\alpha+2}(\partial_u \phi)^2+2\alpha\cdot C_1\cdot[-C_1+o(1)]r^{2\alpha+2}(\partial_v \phi)^2\\
&+2\cdot o(1)\cdot C_2 \cdot r^{2\alpha+2}(\partial_u \phi)^2+2\cdot o(1) \cdot C_1\cdot r^{2\alpha+2}(\partial_v \phi)^2\\
&+2\cdot o(1)\cdot C_1\cdot r^{2\alpha+2}\partial_u \phi \partial_v \phi+2\cdot o(1)\cdot C_2 \cdot r^{2\alpha+2}\partial_u \phi \partial_v \phi\\
\leq& 0, \text{ for sufficiently small } o(1) \text{ and any fixed } \alpha>0.
\end{split}
\end{equation*}

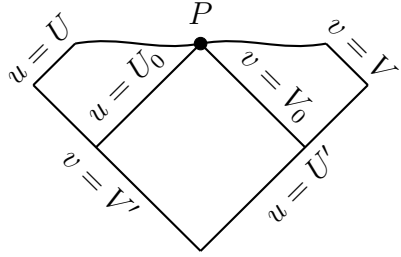
\begin{center}
\begin{tikzpicture}[scale=0.55]
\draw [white](-3, 0.2)-- node[midway, sloped, above,black]{$P$}(3, 0.2);
\draw [white](-4, -1)-- node[midway, sloped, above,black]{$u=U$}(-3, 0);
\draw [white](3, 0)-- node[midway, sloped, above,black]{$v=V$}(4, -1);
\draw [white](0, -5)-- node[midway, sloped, below,black]{$u=U'$}(4, -1);
\draw [white](0, -5)-- node[midway, sloped, below,black]{$v=V'$}(-4, -1);
\draw [white](0, 0)-- node[midway, sloped, above, black]{$u=U_0$}(-2.8,-2.8);
\draw [white](0, 0)-- node[midway, sloped, above, black]{$v=V_0$}(2.8,-2.8);
\draw[thick] (0,0)--(-2.5,-2.5);
\draw[thick] (0,0)--(2.5, -2.5);
\draw [thick] (-3,0) to [out=10,in=-170] (0,0);
\draw [thick] (0,0) to [out=10,in=-170] (3,0);
\draw [thick] (-4,-1)--(-3,0);
\draw [thick] (3,0)--(4,-1);
\draw [thick] (-4,-1)--(0,-5);
\draw [thick] (4,-1)--(0,-5);
\draw[fill] (0,0) circle [radius=0.15];
\end{tikzpicture}
\end{center}
Rewrite $$\iint_{D_0} \l \partial_v [C_2\cdot r^{2\alpha}(r^2\partial_u \phi)^2]+\partial_u [C_1\cdot r^{2\alpha}(r^2\partial_v \phi)^2] \r du dv\leq 0,$$ and we have

\begin{equation*}
\begin{split}
&\int_{v=V'}^{v=V_0}C_1\cdot(r^{2+\alpha}\partial_v  \phi)^2 (U_0, v)dv+\int_{u=U'}^{u=U_0}C_2\cdot(r^{2+\alpha}\partial_u  \phi)^2 (u, V_0)du\\
\leq& \int_{v=V'}^{v=V_0}C_1\cdot(r^{2+\alpha}\partial_v \phi)^2 (U', v)dv+\int_{u=U'}^{u=U_0}C_2\cdot(r^{2+\alpha}\partial_u \phi)^2 (u,V')du\\
=& \mbox{initial data}=I_0^2.
\end{split}
\end{equation*}
Note that equation (\ref{eqn phi}) is equivalent to
\begin{equation}
\partial_u(r\partial_v \phi)=-\partial_v r \partial_u \phi,
\end{equation}
and
\begin{equation}
\partial_v(r\partial_u \phi)=-\partial_u r \partial_v \phi.
\end{equation}
At $P=(U_0, V_0)$, we have
\begin{equation}
\begin{split}
r\partial_v \phi (U_0, V_0)=r\partial_v \phi (U', V_0)-\int_{U'}^{U_0}\partial_v r \partial_u \phi(u, V_0)du.
\end{split}
\end{equation}
This gives
\begin{equation}
\begin{split}
&|r\partial_v \phi (U_0, V_0)|\\
\leq& |r\partial_v \phi (U', V_0)|+|\int_{U'}^{U_0}\partial_v r \partial_u \phi(u, V_0)du|\\
\leq&|r\partial_v \phi (U', V_0)|+\bigg(\int_{U'}^{U_0}(r^{2+\alpha}\partial_u \phi)^2(u, V_0)du\bigg)^{\frac12}\bigg(\int_{U'}^{U_0}(\frac{\partial_v r}{r^{2+\alpha}})^2  (u,V_0)du \bigg)^{\frac12}\\
\leq& \text{ const}+\f{I_0}{\sqrt{C_2}}\cdot \bigg(\int_{U'}^{U_0}(\frac{\partial_v r}{r^{2+\alpha}})^2  (u,V_0)du\bigg)^{\frac12}\\
\lesssim& \text{ const}+\frac{I_0}{r^{2+\alpha}(U_0,V_0)}.
\end{split}
\end{equation}
For the last step, we use
\begin{equation}
\begin{split}
&\f{I_0}{\sqrt{C_2}}\cdot\bigg(\int_{U'}^{U_0}(\frac{\partial_v r}{r^{2+\alpha}})^2  (u,V_0)du\bigg)^{\frac12}\\
\leq&\f{I_0}{\sqrt{C_2}}\cdot\bigg(\int_{U'}^{U_0}(\frac{(\partial_v r)^2}{r^{4+2\alpha} r_u})r_u  (u,V_0)du\bigg)^{\frac12}\\
=&\f{I_0}{\sqrt{C_2}}\cdot\bigg(\int_{r(U',V_0)}^{r(U_0,V_0)}(\frac{(\partial_v r)^2}{r^{4+2\alpha} r_u})dr\bigg)^{\frac12}\approx \f{I_0}{\sqrt{C_2}}\cdot\bigg(\int_{r(U',V_0)}^{r(U_0,V_0)}(\frac{-C_2^2}{C_1 r^{5+2\alpha}})dr\bigg)^{\frac12}\\
\approx& \f{I_0}{\sqrt{C_2}}\cdot\big(\frac{C_2^2}{C_1 r^{4+2\alpha}}(U_0,V_0)\big)^{\f12}\approx \frac{I_0}{r^{2+\alpha}(U_0,V_0)}.
\end{split}
\end{equation}
Hence, we have for $0<\alpha<1$
$$|\partial_v \phi|\leq \f{I_0}{r^{3+\alpha}}.$$
Similarly,
$$|\partial_u \phi|\leq \f{I_0}{r^{3+\alpha}}.$$
\end{proof}

\section{Refined Estimates of $r\partial_u r$ and $r\partial_v r$}

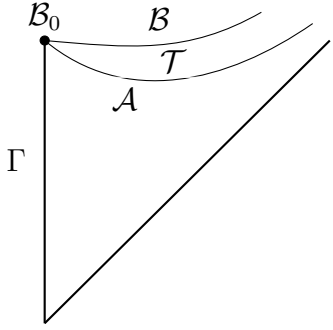
\begin{minipage}[!t]{0.4\textwidth}
\begin{tikzpicture}[scale=0.75]
\draw [white](-1, -2.5)-- node[midway, sloped, above,black]{$\Gamma$}(0, -2.5);
\draw [white](0, 0)-- node[midway, sloped, above,black]{$\mathcal{B}$}(4, 0);
\draw [white](0, -0.75)-- node[midway, sloped, above,black]{$\mathcal{T}$}(4.5, -0.75);
\draw [white](-1, 0)-- node[midway, sloped, above,black]{$\mathcal{B}_0$}(1, 0);
\draw (0,0) to [out=-5, in=210] (3.8, 0.5);
\draw (0,0) to [out=-40, in=215] (4.7, 0.3);
\draw [white](0, -0.65)-- node[midway, sloped, below,black]{$\mathcal{A}$}(2.8, -0.65);

\draw [thick] (0, -5)--(0,0);
\draw [thick] (5, 0)--(0,-5);
\draw[fill] (0,0) circle [radius=0.08];
\end{tikzpicture}
\end{minipage}
\begin{minipage}[!t]{0.6\textwidth}
In \cite{Chr.1}, for general gravitational collapse, we have the Penrose diagram on the left. The curve marked with $\mathcal{A}$ is the apparent horizon.\\

To study the singular boundary $\mathcal{B}$ (where $r=0$), we consider the diamond region (in $\mathcal{T}$) below. The rectangular region is called $D_0$.
\end{minipage}
\hspace{0.05\textwidth}

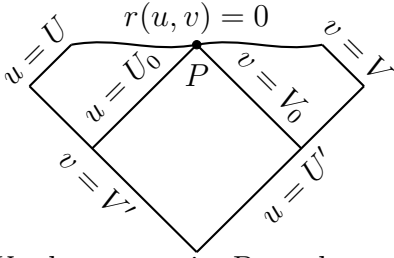
\begin{minipage}[!t]{0.4\textwidth}
\begin{tikzpicture}[scale=0.55]
\draw [white](-3, -0.2)-- node[midway, sloped, below,black]{$P$}(3, -0.2);
\draw [white](-3, 0)-- node[midway, sloped, above,black]{$r(u,v)=0$}(3, 0);
\draw [white](-4, -1)-- node[midway, sloped, above,black]{$u=U$}(-3, 0);
\draw [white](3, 0)-- node[midway, sloped, above,black]{$v=V$}(4, -1);
\draw [white](0, -5)-- node[midway, sloped, below,black]{$u=U'$}(4, -1);
\draw [white](0, -5)-- node[midway, sloped, below,black]{$v=V'$}(-4, -1);
\draw [white](0, 0)-- node[midway, sloped, above, black]{$u=U_0$}(-2.8,-2.8);
\draw [white](0, 0)-- node[midway, sloped, above, black]{$v=V_0$}(2.8,-2.8);
\draw[thick] (0,0)--(-2.5,-2.5);
\draw[thick] (0,0)--(2.5, -2.5);
\draw [thick] (-3,0) to [out=10,in=-170] (0,0);
\draw [thick] (0,0) to [out=10,in=-170] (3,0);
\draw [thick] (-4,-1)--(-3,0);
\draw [thick] (3,0)--(4,-1);
\draw [thick] (-4,-1)--(0,-5);
\draw [thick] (4,-1)--(0,-5);
\draw[fill] (0,0) circle [radius=0.1];
\end{tikzpicture}
\end{minipage}
\hspace{0.05\textwidth}
\begin{minipage}[!t]{0.55\textwidth}
Take $D_0$ to be sufficiently small. By Proposition \ref{Prop 4.1} and continuity, in $D_0$ we have
$$r\partial_u r+C_1=o(1), \quad r\partial_v r+C_2=o(1).\\$$

To derive sharp blow-up rates for $\partial_u \phi$ and $\partial_v \phi$, we will need improved estimates for $r\partial_u r$ and $r\partial_v r$.
\end{minipage}

We then zoom in $D_0$ and assume $U_0=0, V_0=0$. 

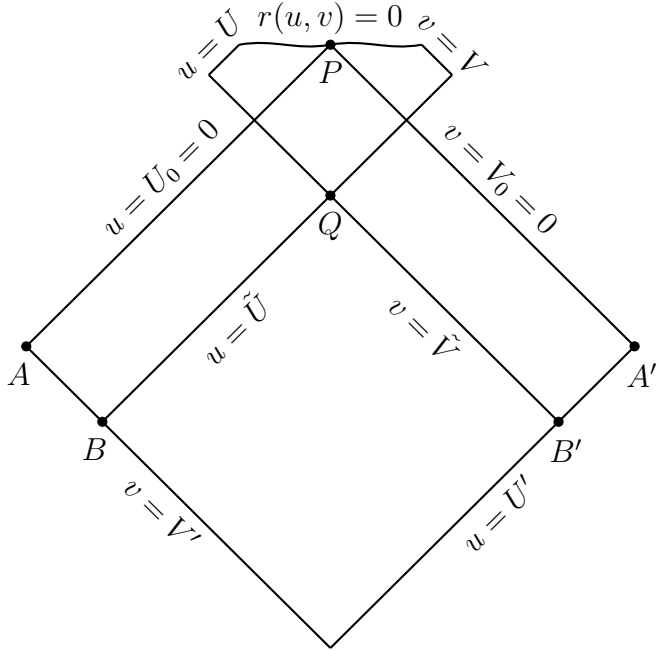
\begin{minipage}[!t]{0.4\textwidth}
\begin{tikzpicture}[scale=0.4]
\draw [white](-3, -0.2)-- node[midway, sloped, below,black]{$P$}(3, -0.2);
\draw [white](-3, -5.2)-- node[midway, sloped, below,black]{$Q$}(3, -5.2);
\draw [white](-13.5, -10.2)-- node[midway, sloped, below,black]{$A$}(-7, -10.2);
\draw [white](-6, -12.7)-- node[midway, sloped, below,black]{$B$}(-9.5, -12.7);
\draw [white](13.5, -10.2)-- node[midway, sloped, below,black]{$A'$}(7, -10.2);
\draw [white](6, -12.7)-- node[midway, sloped, below,black]{$B'$}(9.5, -12.7);
\draw [white](-3, 0)-- node[midway, sloped, above,black]{$r(u,v)=0$}(3, 0);
\draw [white](-4, -1)-- node[midway, sloped, above,black]{$u=U$}(-3, 0);
\draw [white](3, 0)-- node[midway, sloped, above,black]{$v=V$}(4, -1);
\draw [white](0, -5)-- node[midway, sloped, below,black]{$u=\tilde{U}$}(-7.5, -12.5);
\draw [white](0, -5)-- node[midway, sloped, below,black]{$v=\tilde{V}$}(7.5, -12.5);
\draw [white](0, -20)-- node[midway, sloped, below,black]{$u=U'$}(10, -10);
\draw [white](0, -20)-- node[midway, sloped, below,black]{$v=V'$}(-10, -10);
\draw [white](0, 0)-- node[midway, sloped, above, black]{$u=U_0=0$}(-10,-10);
\draw [white](0, 0)-- node[midway, sloped, above, black]{$v=V_0=0$}(10,-10);
\draw[thick] (0,0)--(-10,-10);
\draw[thick] (0,0)--(10, -10);
\draw[thick] (0,-5)--(-7.5,-12.5);
\draw[thick] (0,-5)--(7.5,-12.5);
\draw[thick] (-10, -10)--(-7.5, -12.5);
\draw[thick] (10, -10)--(7.5, -12.5);
\draw[thick] (-7.5, -12.5)--(0, -20);
\draw[thick] (7.5, -12.5)--(0, -20);
\draw [thick] (-3,0) to [out=10,in=-170] (0,0);
\draw [thick] (0,0) to [out=10,in=-170] (3,0);
\draw [thick] (-4,-1)--(-3,0);
\draw [thick] (3,0)--(4,-1);
\draw [thick] (-4,-1)--(0,-5);
\draw [thick] (4,-1)--(0,-5);
\draw[fill] (0,0) circle [radius=0.15];
\draw[fill] (0,-5) circle [radius=0.15];
\draw[fill] (-10,  -10) circle [radius=0.15];
\draw[fill] (-7.5,-12.5) circle [radius=0.15];
\draw[fill] (10,  -10) circle [radius=0.15];
\draw[fill] (7.5,-12.5) circle [radius=0.15];
\end{tikzpicture}
\end{minipage}
\hspace{0.05\textwidth}
\begin{minipage}[!t]{0.55\textwidth}
\end{minipage}

\noindent In the image above, we have $U, V, \tilde{U}, \tilde{V}, U', V' < 0$. We are now ready to state and prove
\begin{proposition}\label{improved ru rv}
For $Q\in D_0$ sufficiently close to $P$, we have improved estimates
\begin{equation}
|r\partial_u r+C_1|(Q)\leq 2r(Q)^{\f{1}{100}}, \quad |r\partial_v r+C_2|(Q)\leq 2r(Q)^{\f{1}{100}}.
\end{equation}
\end{proposition}
\begin{proof}
Denote $r(Q)=r_0$. Along $u=0$, we first find $A$ in the past of $P$ and satisfying
$$r(A)=r(Q)^{\f{1}{100}}=r_0^{\f{1}{100}}.$$
Assume $B$ is the intersection of $v=V'$ and $u=\tilde{U}$. When $Q$ is sufficiently close to $P$. $A$ and $B$ are still in the region $D_0$. And in $D_0$, we have
$$r\partial_u r+C_1=o(1), \quad r\partial_v r+C_2=o(1).$$
This implies
\begin{equation*}
\begin{split}
r^2(u,0)=&r^2(0,0)+\int_0^u \partial_u (r^2(u',0))du'=\int_0^u (2r\partial_u r) (u',0)du'\\
=&\int_0^u[-2C_1+o(1)]du'=[-2C_1+o(1)]u.
\end{split}
\end{equation*}
Similarly, we further have
\begin{equation}\label{asympupperbdofrsq}
\begin{split}
r^2(u,v)=&r^2(u,0)+\int_0^v \partial_v (r^2(u, v'))dv'=r^2(u,0)+\int_0^v (2r\partial_v r) (u,v')dv'\\
=&[-2C_1+o(1)]u+\int_0^v [-2C_2+o(1)]dv'\\
=&[-2C_1+o(1)]u+[-2C_2+o(1)]v.
\end{split}
\end{equation}
In particular, at $Q$, where $u=\tilde{U}, v=\tilde{V}$, it holds
$$[-2C_1+o(1)]\tilde{U}+[-2C_2+o(1)]\tilde{V}=r^2_0.$$
This implies
$$|\tilde{U}|\leq \f{r^2_0}{C_1}, \quad \quad \quad \quad |\tilde{V}|\leq \f{r^2_0}{C_2}.$$
Along $AP$, we use equation
$$\partial_v (r\partial_u r)=-\f14\O^2$$
and taking (\ref{Omega upper}) into account, we obtain
\begin{equation}\label{PA}
\begin{split}
|(r\partial_u r)(P)-(r\partial_u r)(A)|=&\int_{V'}^0\f14\O^2(0,v')dv'\lesssim \int_{V'}^0\f{D}{r(0,v')}dv'\\
\approx& \int_{V'}^0 -\partial_v r(0,v')dv' = r(A)=r_0^{\f{1}{100}}.
\end{split}
\end{equation}
Along $BA$, from (\ref{eqn u}) we have
$$\partial_u(\partial_u r)=-r(\partial_u \phi)^2+2\partial_u \log \O\cdot \partial_u r.$$
Hence
\begin{equation}\label{ruu}
\begin{split}
\partial_u (r\partial_u r)=&r\partial_u (\partial_u r)+\partial_u r\cdot \partial_u r\\
=&-r^2(\partial_u \phi)^2+2\partial_u \log\O\cdot r\cdot \partial_u r+\partial_u r \partial_u r
\end{split}
\end{equation}
Using Proposition \ref{preliminary}, we have $r^2(\partial_u \phi)^2+|r^2\partial_u\phi\partial_v \phi|\leq {I^2_0}/{r^{4+2\a}}.$
Proposition \ref{Prop 4.1} gives $|\partial_u r|, |\partial_v r|\lesssim 1/r$. Hence via (\ref{eqn Omega}):
$$r^2\partial_u \partial_v \log \O=\partial_u r \partial_v r+\f14\O^2-r^2 \partial_u \phi \partial_v \phi$$
and (\ref{Omega upper}), taking (\ref{asympupperbdofrsq}) into account and integrating (\ref{eqn Omega}) we obtain $|\partial_u\log\O|\leq I^2_0/r^{4+2\a}$. With these estimates, we bound the RHS of (\ref{ruu}) and obtain
\begin{equation}\label{AB}
\begin{split}
|(r\partial_u r)(A)-(r\partial_u r)(B)|\leq& \int_{\tilde{U}}^0 \f{I^2_0}{r^{{\color{black} 4}+2\a}}(u',V')du'\lesssim \f{1}{r^{4+2\a}(A)}\cdot |\tilde{U}|\\
\leq&\f{1}{r_0^{\f{4+2\a}{100}}}\cdot \f{r_0^2}{C_1}\leq r_0.
\end{split}
\end{equation}
Combining (\ref{PA}) and (\ref{AB}), we get
\begin{equation}\label{PB}
|(r\partial_u r)(P)-(r\partial_u r)(B)|\leq r_0^{\f{1}{100}}.
\end{equation}
Along $AB$, using $r\partial_u r \sim \text{const}$ we have
$$|\partial_u r|\lesssim \f{1}{r(A)}=\f{1}{r_0^{ \f{1}{100}}}.$$
Thus, it follows
$$|r(B)-r(A)|\leq \int_{\tilde{U}}^0 |\partial_u r(u', V')|du'\leq \f{1}{r_0^{\f{1}{100}}}\cdot \f{r_0^2}{C_1}\leq r_0.$$
Hence, for $r(B)$ we have
$$r_0^{\f{1}{100}}\leq r(A) \leq r(B) \leq r(A)+r_0\leq 2 r_0^{\f{1}{100}}.$$
Lastly, along $BQ$, we use equation $\partial_v (r\partial_u r)=-\f{1}{4}\O^2$
and obtain
\begin{equation}\label{BQ}
\begin{split}
|(r\partial_u r)(B)-(r\partial_u r)(Q)|=&\int_{V'}^{\tilde{V}}\f14\O^2(\tilde{U},v')dv'\lesssim \int_{V'}^{\tilde{V}}\f{D}{r(\tilde{U},v')}dv'\\
\approx& \int_{V'}^{\tilde{V}} -\partial_v r(\tilde{U},v')dv' = r(B)-r(Q)\leq 2r_0^{\f{1}{100}}.
\end{split}
\end{equation}
Combining (\ref{PB}) and (\ref{BQ}), we then obtain
\begin{equation}\label{PQ}
|(r\partial_u r)(P)-(r\partial_u r)(Q)|\leq 2r_0^{\f{1}{100}}.
\end{equation}
This gives
$$|r\partial_u r+C_1|(Q)\leq 2r(Q)^{\f{1}{100}}.$$
Similarly, by using $A'$ and $B'$, we have
$$|r\partial_v r+C_2|(Q)\leq 2r(Q)^{\f{1}{100}}.$$
Note that this conclusion holds for all the points $Q$ sufficiently close to $P$.
\end{proof}

\section{Sharp Estimates for $\partial_u \phi$ and $\partial_v \phi$}\label{section sharp phi}

We are now ready to prove:

\noindent Theorem \ref{prop1.2}.
Under the same assumptions as in Theorem \ref{thm1.1}, at any point {\color{black}$(u,v)\in\mathcal{T}$ near $\mathcal{B}$}, there exists positive number $D_1$ and $D_2$ (depending on the initial data), such that
$$|\partial_u \phi(u,v)|\leq \f{D_1}{r(u,v)^2}, \quad \quad |\partial_v \phi(u,v)|\leq \f{D_2}{r(u,v)^2}.$$

\begin{proof}
We consider the spacetime region (in $D_0$) below. Fix $l \gg 1$ so that the entire figure below is in our region of interest. Let $n \gg l$ be arbitrary.

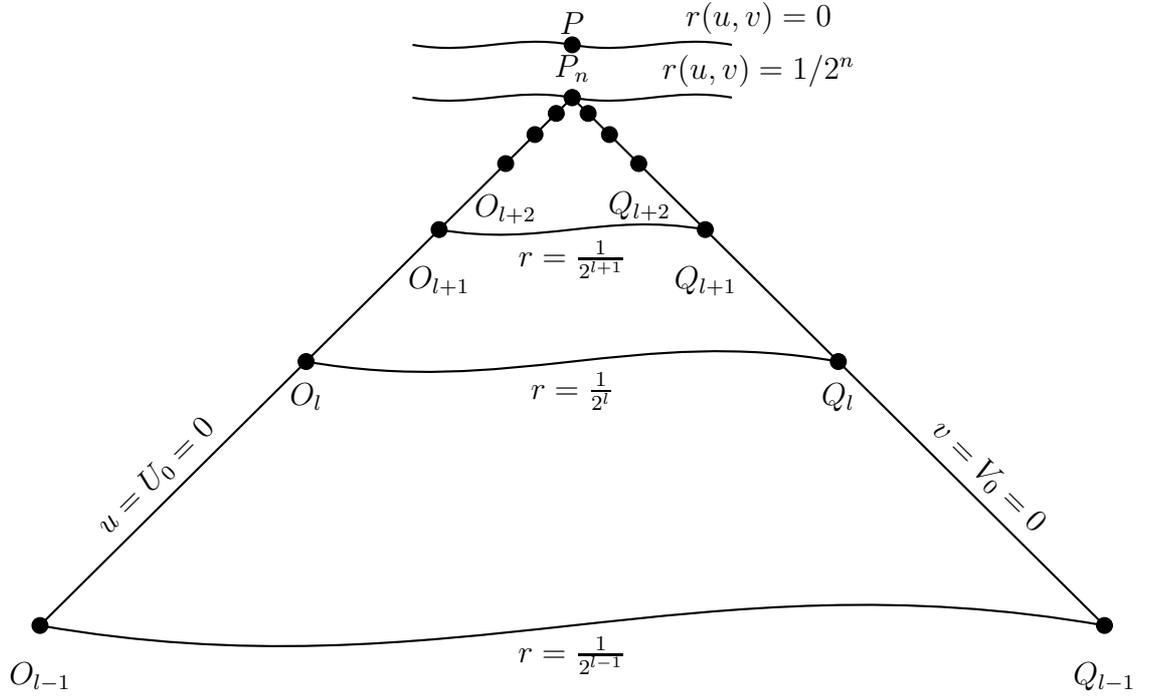
\begin{minipage}[!t]{0.4\textwidth}
\begin{tikzpicture}[scale=0.7]
\draw [white](4, 1)-- node[midway, sloped, above,black]{$r(u,v)=0$}(3, 1);
\draw [white](-2, 1)-- node[midway, sloped, above,black]{$P$}(2, 1);
\draw [white](-2, 0.1)-- node[midway, sloped, above,black]{$P_n$}(2, 0.1);
\draw [white](-2.51, -3)-- node[midway, sloped, below,black]{$O_{l+1}$}(-2.5, -3);
\draw [white](-6, -5.2)-- node[midway, sloped, below,black]{$O_l$}(-4, -5.2);
\draw [white](-11, -10.5)-- node[midway, sloped, below,black]{$O_{l-1}$}(-9, -10.5);
\draw [white](-1.26, -1.63)-- node[midway, sloped, below,black]{$O_{l+2}$}(-1.25, -1.63);
\draw [white](2.51, -3)-- node[midway, sloped, below,black]{$Q_{l+1}$}(2.5, -3);
\draw [white](6, -5.2)-- node[midway, sloped, below,black]{$Q_l$}(4, -5.2);
\draw [white](11, -10.5)-- node[midway, sloped, below,black]{$Q_{l-1}$}(9, -10.5);
\draw [white](1.28, -1.57)-- node[midway, sloped, below,black]{$Q_{l+2}$}(1.25, -1.57);
\draw [white](4, 0)-- node[midway, sloped, above,black]{$r(u,v)=1/2^n$}(3, 0);

\draw [white](-5, -5)-- node[midway, sloped, above, black]{$u=U_0=0$}(-10,-10);
\draw [white](5, -5)-- node[midway, sloped, above, black]{$v=V_0=0$}(10,-10);
\draw[thick] (0,0)--(-10,-10);
\draw[thick] (0,0)--(10, -10);
\draw [thick] (-3,0) to [out=-10,in=170] (0,0);
\draw [thick] (0,0) to [out=-10,in=170] (3,0);
\draw [thick] (-3,1) to [out=-10,in=170] (0,1);
\draw [thick] (0,1) to [out=-10,in=170] (3,1);
\draw [white](-10, -10)-- node[midway, sloped, below,black]{$r=\f{1}{2^{l-1}}$}(10, -10);
\draw [white](-5, -5)-- node[midway, sloped, below,black]{$r=\f{1}{2^{l}}$}(5, -5);
\draw [white](-2.5, -2.5)-- node[midway, sloped, below,black]{$r=\f{1}{2^{l+1}}$}(2.5, -2.5);
\draw[thick] (-10, -10) to [out=-10, in=170] (10, -10);
\draw[thick] (-5, -5) to [out=-10, in=170] (5, -5);
\draw[thick](-2.5, -2.5) to [out=-10, in=170] (2.5, -2.5);
\draw[fill] (0,0) circle [radius=0.15];
\draw[fill] (0,1) circle [radius=0.15];
\draw[fill] (0,0) circle [radius=0.15];
\draw[fill] (-10,-10) circle [radius=0.15];
\draw[fill] (-5,-5) circle [radius=0.15];
\draw[fill] (-2.5,-2.5) circle [radius=0.15];
\draw[fill] (-1.25,-1.25) circle [radius=0.15];
\draw[fill] (-0.7,-0.7) circle [radius=0.15];
\draw[fill] (-0.3,-0.3) circle [radius=0.15];
\draw[fill] (10,-10) circle [radius=0.15];
\draw[fill] (5,-5) circle [radius=0.15];
\draw[fill] (2.5,-2.5) circle [radius=0.15];
\draw[fill] (1.25,-1.25) circle [radius=0.15];
\draw[fill] (0.7,-0.7) circle [radius=0.15];
\draw[fill] (0.3,-0.3) circle [radius=0.15];
\end{tikzpicture}
\end{minipage}
\hspace{0.05\textwidth}
\begin{minipage}[!t]{0.6\textwidth}
\end{minipage}

\noindent In $D_0$, we consider different constant $r${\color{green}-}level sets $\{L_r\}$. Let
$$\Psi(r)=\max\{\sup_{P\in L_r} |C_2\cdot r\partial_u \phi|(P), \sup_{Q\in L_r} |C_1\cdot r\partial_v \phi|(Q)\}.$$
At $P$, we have
$$-r\partial_u r(P)=C_1>0, \quad -r\partial_v r(P)=C_2>0.$$
Then from (\ref{eqn phi}), i.e.
$$\partial_{u}(r\partial_v \phi)=-r_v \partial_u\phi$$
we have
\begin{equation}
\begin{split}
|C_1\cdot r\partial_v \phi|(P_n)\leq& I.D.+ \int_{u(Q_{l-1})}^{u(P_n)} -r_v |C_1\cdot\partial_u\phi| \, du\\
\leq& I.D.+ \int_{u(Q_{l-1})}^{u(P_n)} -r_u\cdot \f{r \partial_v r}{r \partial_u r} |C_1\cdot\partial_u\phi| \, du\\
=& I.D. +\int_{r(Q_{l-1})}^{r(P_n)} -\f{r \partial_v r}{r\partial_u r}\cdot\f{C_1}{C_2}\cdot\f{1}{r}\cdot |C_2\cdot r\partial_u\phi| \, dr\\
=& I.D. +\int_{r(Q_{l-1})}^{r(P_n)} -\f{1+O(r^{\f{1}{100}})}{r}\cdot |C_2\cdot r\partial_u\phi| \, dr \quad \quad (\mbox{use Proposition }\ref{improved ru rv})
\end{split}
\end{equation}
Similarly, we have
$$|C_2\cdot r\partial_u \phi|(P_n)\leq I.D. +\int_{r(O_{l-1})}^{r(P_n)} -\f{1+O(r^{\f{1}{100}})}{r}\cdot |C_1\cdot r\partial_v\phi| \, dr.$$
Combining these two inequality together, we have
$$\Psi(2^{-n})\leq I.D.+\int_{r=2^{-l+1}}^{r=2^{-n}} -\f{1+O(r^{\f{1}{100}})}{r}\cdot \Psi(r) \, dr.$$
Here $2^{-n}$ could be replaced by any small positive number. Hence it is true that for any small enough $\tilde{r}>0$
$$\Psi(\tilde{r})\leq I.D.+\int_{2^{-l+1}}^{\tilde{r}} -\f{1+O(r^{\f{1}{100}})}{r}\cdot \Psi(r) \, dr=I.D.+\int^{2^{-l+1}}_{\tilde{r}} \f{1+O(r^{\f{1}{100}})}{r}\cdot \Psi(r) \, dr$$
{\color{black}By} Gr\"onwall's inequality, we have
$$\Psi(\tilde{r})\leq I.D.\times e^{\int_{\tilde{r}}^{2^{-l+1}} \f{1+O(r^{\f{1}{100}})}{r}\cdot} dr=I.D.\times e^{-\ln \tilde{r}+O(1)}\leq\f{C}{\tilde{r}},$$
where $C$ is a uniform number depending on initial data. This gives
$$\tilde{r}\Psi(\tilde{r})\leq C \mbox{ for any } \tilde{r}>0,$$
which further implies
$$r^2|\partial_u \phi|\leq D_1, \quad r^2|\partial_v \phi|\leq D_2$$
for any $r\geq 0$, where $D_1, D_2$ are uniform numbers depending only on initial data.
\end{proof}

\section{Higher Order Estimates}
For the purpose of future use, we first state several useful estimates.
\begin{proposition}\label{1stderivofOmega}
For $\O^2(u,v)$, we have
$$|\O^2(u,v)|\lesssim \f{1}{r(u,v)}, \quad |\partial_u\log\O|(u,v)\lesssim\f{1}{r^2(u,v)}, \quad  |\partial_v\log\O|(u,v)\lesssim\f{1}{r^2(u,v)}.$$
$$|\partial_u (\O^2(u,v))|\lesssim \f{1}{r^3(u,v)}, \quad |\partial_v(\O^2(u,v))(u,v)|\lesssim \f{1}{r^3(u,v)}.$$

\end{proposition}
\begin{proof}
For $\O^2(u,v)$, we already proved $|\O^2(u,v)|\ls {1}/{r(u,v)}$ in (\ref{Omega upper}). From (\ref{eqn Omega}) we have
$$\partial_v(\f{1}{\O^2}\partial_u(\O^2))=\partial_v \partial_u \log (\O^2)=\f{1}{2r^2}(\partial_u r \partial_v r+\f14\O^2-r^2\partial_u \phi \partial_v \phi).$$
This implies
\begin{equation*}
\begin{split}
&\f{1}{\O^2}\partial_u(\O^2)(u,v)=\partial_u\log (\O^2)(u,v)\\
=&\f{1}{\O^2}\partial_u(\O^2)(u,v_0)+\int_{v_0}^v \f{1}{2r^2}\big(\partial_u r \partial_v r+\f14\O^2-r^2\partial_u \phi \partial_v \phi\big)(u,v')dv'\\
\lesssim&I.D.+\f{1}{r^2(u,v)}.
\end{split}
\end{equation*}
In the above inequality chain we used Theorem \ref{prop1.2} and (\ref{Omega upper}). Note that the above $I.D.$ is uniformly bounded.
Together with $\O^2(u,v)\lesssim\f{1}{r(u,v)}$ by (\ref{Omega upper}), we conclude
$$|\partial_u \O^2(u,v)|{\color{black}\lesssim} \O^2(u,v)[I.D.+\f{1}{r^2(u,v)}]{\color{black}\lesssim} \f{1}{r^3(u,v)}.$$
Similarly, we have
$$|\partial_v \O^2(u,v)|{\color{black}\lesssim} \f{1}{r^3(u,v)}.$$
\end{proof}

\begin{proposition}\label{2ndorderr}
For $r(u,v)$, we have
$$|\partial_u \partial_v {\color{black}(r^2)}|{\color{black}\lesssim \f{1}{r}\lesssim} \f{1}{r^2}, \quad|\partial_u \partial_u {\color{black}(r^2)}|{\color{black}\lesssim} \f{1}{r^2}, \quad |\partial_v \partial_v {\color{black}(r^2)}|{\color{black}\lesssim} \f{1}{r^2},$$
$$|\partial_u \partial_v r|{\color{black}\lesssim} \f{1}{r^3}, \quad|\partial_u \partial_u r|{\color{black}\lesssim} \f{1}{r^3}, \quad |\partial_v \partial_v r|{\color{black}\lesssim} \f{1}{r^3}.$$
\end{proposition}
\begin{proof}
For $\partial_u(\partial_v {\color{black}(r^2)})$, we have {\color{black}from (\ref{eqn r}):}
$$\partial_u(\partial_v {\color{black}(r^2)})=-\f12\O^2,$$
the desired estimate follows from the derived estimate $\O^2\ls1/r\ls 1/r^2$.

For $\partial_u(\partial_u {\color{black}(r^2)})$, we first recall {\color{black}from (\ref{eqn u}):}
$$-2\partial_u \log\O\cdot \partial_u r+\partial_u(\partial_u r)=-r(\partial_u\phi)^2.$$
This implies
$$-4r\partial_u \log\O\cdot \partial_u r+\partial_u(2r\partial_u r)-2\partial_u r\partial_u r=-2r^2(\partial_u\phi)^2,$$
that is
$$\partial_u(\partial_u {\color{black}(r^2)})=4r\partial_u \log\O\cdot \partial_u r+2\partial_u r\partial_u r-2r^2(\partial_u\phi)^2.$$
By the estimates {\color{black}in Proposition \ref{1stderivofOmega}, (\ref{Omega upper}) and Proposition \ref{prop1.2}}, we have
$$|\partial_u(\partial_u r^2)|{\color{black}\lesssim} \f{1}{r^2}.$$
Similarly, we also have
$$|\partial_v(\partial_v r^2)|{\color{black}\lesssim} \f{1}{r^2}.$$
{\color{black}Hence we have every inequality in the second line of the statement.}
\end{proof}

\begin{proposition}\label{2ndderivofphi}
For $\phi(u,v)$, we have
$$|\partial_u \partial_v \phi|{\color{black}\lesssim} \f{1}{r^4}, \quad |\partial_u \partial_u \phi|{\color{black}\lesssim} \f{1}{r^4}, \quad |\partial_v \partial_v \phi|{\color{black}\lesssim} \f{1}{r^4}.$$
\end{proposition}
\begin{proof}
The first estimate follows from {\color{black}(\ref{eqn phi}):}
$$r\partial_u\partial_v \phi=-\partial_u r\partial_v \phi-\partial_v r\partial_u\phi$$
{\color{black}and Proposition \ref{prop1.2}.}
We then differentiate the above equation with respect to $u$. Rewrite it. We then get
$$\partial_v (r\partial^2_{uu}\phi)=-\partial^2_{uu}r\cdot\partial_v\phi-\partial^2_{uv}r\cdot\partial_u\phi-{\color{black}2\partial_u r \partial^2_{uv}\phi}.$$
{\color{black}By previous bounds, the right hand side of the above inequality is $\lesssim \f{1}{r^4}$.}
Integrate both sides with respect to $v$, we derive
$$|r\partial^2_{uu}\phi|\leq \f{1}{r^3},$$
which implies
$$|\partial^2_{uu}\phi|\leq \f{1}{r^4}.$$
{\color{black}We similarly have the other desired estimate on $\partial^2_{vv}\phi$.}
\end{proof}

\begin{proposition}\label{2ndderivofOmega}
For $\log{\color{black}(\O^2)}(u,v)$ we have
$$|\partial_u\partial_v \log{\color{black}(\O^2)}|{\color{black}\lesssim} \f{1}{r^4}, \quad |\partial_u\partial_u \log{\color{black}(\O^2)}|{\color{black}\lesssim} \f{1}{r^4}, \quad |\partial_v\partial_v \log{\color{black}(\O^2)}|{\color{black}\lesssim} \f{1}{r^4},$$
$$|\partial_v\partial_u {\color{black}(\O^2)}|{\color{black}\lesssim}\f{1}{r^5}, \quad |\partial_u \partial_u {\color{black}(\O^2)}|{\color{black}\lesssim} \f{1}{r^5}, \quad |\partial_v \partial_v {\color{black}(\O^2)}|{\color{black}\lesssim} \f{1}{r^5}.$$
\end{proposition}
\begin{proof}
This first estimate is easily obtained from
$$\partial_v \partial_u \log {\color{black}(\O^2)}=\f{1}{2r^2}(\partial_u r \partial_v r+\f14\O^2-r^2\partial_u \phi \partial_v \phi).$$
Differentiate this equation with respect to $u$ and integrate {\color{black}the result} with respect to $v$, with the help of derived estimates, we arrive at
$$|\partial_u\partial_u \log{\color{black}(\O^2)}|{\color{black}\lesssim} \f{1}{r^4}.$$
Since
$$\f{\partial_u\partial_u {\color{black}(\O^2)}}{\O^2}-\f{\partial_u {\color{black}(\O^2)}\cdot \partial_u \log{\color{black}(\O^2)}}{\O^2}  =\partial_u \f{\partial_u {\color{black}(\O^2)}}{\O^2}=\partial_u\partial_u\log{\color{black}(\O^2)}.$$
That is
$${\partial_u\partial_u {\color{black}(\O^2)}}=\partial_u {\color{black}(\O^2)}\cdot \partial_u \log{\color{black}(\O^2)}+\O^2\cdot\partial_u\partial_u\log{\color{black}(\O^2)}.$$
By the estimates derived above, we have
$$|\partial_u \partial_u {\color{black}(\O^2)}|{\color{black}\lesssim} \f{1}{r^5}.$$
Similarly, we also have
$$|\partial_v\partial_v \log{\color{black}(\O^2)}|{\color{black}\lesssim} \f{1}{r^4}, \mbox{ and } |\partial_v \partial_v {\color{black}(\O^2)}|{\color{black}\lesssim} \f{1}{r^5}.$$
From
$$\f{\partial_v\partial_u{\color{black}(\O^2)}}{\O^2}-\f{\partial_u {\color{black}(\O^2)}\cdot \partial_v \log{\color{black}(\O^2)}}{\O^2}=\partial_v \f{\partial_u \O^2}{\O^2}=\partial_v\partial_u \log{\color{black}(\O^2)},$$
we have
$${\partial_v\partial_u\O^2}={\partial_u \O^2\cdot \partial_v \log\O^2}+\O^2\cdot\partial_v\partial_u \log\O^2.$$
With derived estimates, we have
$$|\partial_v\partial_u {\color{black}(\O^2)}|{\color{black}\lesssim} \f{1}{r^5}.$$

\end{proof}

With these estimates, in the same fashion, with equations
$$\partial_u(\partial_u {\color{black}(r^2)})=4r\partial_u \log\O\cdot \partial_u r+2\partial_u r\partial_u r-2r^2(\partial_u\phi)^2.$$
we then get
$$|\partial_u \partial_u \partial_u {\color{black}(r^2)}|{\color{black}\lesssim} \f{1}{r^4}, \quad |\partial_v \partial_u \partial_u {\color{black}(r^2)}|{\color{black}\lesssim} \f{1}{r^4}.$$
Similarly, we also have
$$|\partial_v \partial_v \partial_v {\color{black}(r^2)}|{\color{black}\lesssim} \f{1}{r^4}, \quad |\partial_u \partial_v \partial_v {\color{black}(r^2)}|{\color{black}\lesssim} \f{1}{r^4}.$$
Repeatedly, we can derive all desired estimates through the following order
\begin{equation}
\begin{split}
&\partial_u\log{\color{black}(\O^2)}, \partial_v{\color{black}(\O^2)}, \partial_u {\color{black}(\O^2)}, \partial_v {\color{black}(\O^2)}\\
\rightarrow& r_{uu}, r_{vv}, r_{uv}, (r^2)_{uu}, (r^2)_{vv}, (r^2)_{uv}\\
\rightarrow& \partial^2_{uu}\phi, \partial^2_{vv}\phi, \partial^2_{uv}\phi\\
\rightarrow& \partial_{uu}\log{\color{black}(\O^2)}, \partial_{vv}\log{\color{black}(\O^2)}, \partial_{uv}\log{\color{black}(\O^2)}, \partial_{uu} {\color{black}(\O^2)}, \partial_{vv} {\color{black}(\O^2)}, \partial_{uv} {\color{black}(\O^2)}\\
\rightarrow& r_{uuu}, r_{uuv}, r_{uvv}, r_{vvv}, (r^2)_{uuu}, (r^2)_{uuv}, (r^2)_{uvv}, (r^2)_{vvv}\\
\rightarrow& \partial^3_{uuu}\phi, \partial^3_{uuv}\phi, \partial^3_{uvv}\phi, \partial^3_{vvv}\phi\\
\rightarrow&......
\end{split}
\end{equation}
In particular, {\color{black}in $D_0$} we get
\begin{proposition}
For any $m,n\in \mathbb{N}$, we have
{\color{black}$$|(\partial_u)^m(\partial_v)^n (r^2)(u,v)|\leq \f{1}{r^{2m+2n{\color{black}-2}}(u,v)}$$
and
$$|(\partial_u)^m(\partial_v)^n \log (\Omega^2)(u,v)|\leq \f{1}{r^{2m+2n}(u,v)}$$}
\end{proposition}
\noindent Recall in $D_0$, we have
$$r^2(u,v)=[-2C_1+o(1)]u+[-2C_2+o(1)]v.$$
This further implies, {\color{black} in $D_0$ when it is close to a singular boundary point with coordinate $(u,v)=(0,0)$}
\begin{proposition}
For any $m,n\in \mathbb{N}$, we have
$$|(u\partial_u)^m(v\partial_v)^n r^2(u,v)|\lesssim {\color{black}r^2}.$$
\end{proposition}

{\color{black}

\section{Estimates of Kretschmann scalar}\label{Kretschmann}
By direct calculation, for Christoffel {\color{black}symbols} of metric (\ref{metric}) we have
$$\Gamma^u_{uu}=\f{2\partial_u \O}{\O}, \quad \Gamma^u_{\theta \theta}=\f{2r\partial_v r}{\O^2}, \quad\Gamma^u_{\phi \phi}=\f{2r \sin^2\theta\cdot \partial_v r}{\O^2},$$
$$\Gamma^v_{vv}=\f{2\partial_v \O}{\O},\quad \Gamma^v_{\theta \theta}=\f{2r\partial_u r}{\O^2}, \quad \Gamma^v_{\phi \phi}=\f{2r \sin^2\theta\cdot \partial_u r}{\O^2},$$
$$\Gamma^{\theta}_{u\theta}=\f{\partial_u r}{r},\quad \Gamma^{\theta}_{v\theta}=\f{\partial_v r}{r}, \quad \Gamma^{\theta}_{\phi \phi}=-\sin \theta\cdot \cos\theta,$$
$$\Gamma^{\phi}_{u\phi}=\f{\partial_u r}{r}, \quad \Gamma^{\phi}_{v\phi}=\f{\partial_v r}{r}, \quad \Gamma^{\phi}_{\theta \phi}=\f{\cos \theta}{\sin \theta.}$$
We then write down the expression for {\color{black}the} Kretschmann scalar:
\begin{equation}\label{Kretschmann2}
\begin{split}
&R^{\alpha\beta\rho\sigma}R_{\alpha\beta\rho\sigma}\\
=&\f{4}{r^4 \O^8}\bigg(16\cdot(\f{\partial^2 r}{\partial u \partial v})^2\cdot r^2\cdot\O^4+16\cdot \f{\partial^2 r}{\partial u^2}\cdot\f{\partial^2 r}{\partial v^2} \cdot r^2\cdot\O^4 \bigg)\\
&+\f{4}{r^4 \O^8}\bigg(-32\cdot \f{\partial^2 r}{\partial u^2}\cdot\partial_v r \cdot r^2 \cdot \O^3\cdot \partial_v \O -32\cdot\f{\partial^2 r}{\partial v^2}\cdot r^2\cdot\partial_u r \cdot \O^3\cdot \partial_u \O \bigg)\\
&+\f{4}{r^4 \O^8}\bigg( 16\cdot (\partial_v r)^2\cdot (\partial_u r)^2\cdot \O^4+64\cdot \partial_v r\cdot r^2\cdot \partial_u r\cdot \O^2\cdot \partial_u \O\cdot \partial_v\O+8\cdot\partial_v r\cdot \partial_u r\cdot \O^6\bigg)\\
&+\f{4}{r^4 \O^8}\bigg(  16\cdot r^4\cdot (\f{\partial^2 \O}{\partial v \partial u})^2\cdot \O^2-32\cdot r^4\cdot \f{\partial^2 \O}{\partial v \partial u}\cdot \O \cdot \partial_v \O \cdot \partial_u \O \bigg)\\
&+\f{4}{r^4 \O^8}\bigg( 16\cdot r^4\cdot (\partial_v \O)^2\cdot (\partial_u \O)^2+\O^8 \bigg)\\
=&\f{4}{r^4 \O^8}\bigg(16\cdot(\f{\partial^2 r}{\partial u \partial v})^2\cdot r^2\cdot\O^4+16\cdot \f{\partial^2 r}{\partial u^2}\cdot\f{\partial^2 r}{\partial v^2} \cdot r^2\cdot\O^4 \bigg)\\
&+\f{4}{r^4 \O^8}\bigg(-32\cdot \f{\partial^2 r}{\partial u^2}\cdot\partial_v r \cdot r^2 \cdot \O^4\cdot \partial_v \log\O -32\cdot\f{\partial^2 r}{\partial v^2}\cdot r^2\cdot\partial_u r \cdot \O^4\cdot \partial_u \log\O \bigg)\\
&+\f{4}{r^4 \O^8}\bigg( 16\cdot (\partial_v r)^2\cdot (\partial_u r)^2\cdot \O^4+64\cdot \partial_v r\cdot r^2\cdot \partial_u r\cdot \O^4\cdot \partial_u \log\O\cdot \partial_v\log\O+8\cdot\partial_v r\cdot \partial_u r\cdot \O^6\bigg)\\
&+\f{4}{r^4 \O^8}\bigg(16\cdot r^4\cdot (\f{\partial^2 \O}{\partial v \partial u})^2\cdot \O^2-32\cdot r^4\cdot \f{\partial^2 \O}{\partial v \partial u}\cdot \O^3 \cdot \partial_v \log\O \cdot \partial_u \log\O \bigg)\\
&+\f{4}{r^4 \O^8}\bigg( 16\cdot r^4\cdot \O^4\cdot (\partial_v \log\O)^2\cdot (\partial_u \log\O)^2+\O^8 \bigg).\\
\end{split}
\end{equation}
By Proposition \ref{Prop 4.1} and Proposition \ref{2ndorderr}, we obtain polynomial upper bounds for $|\partial_u r|, |\partial_v r|, |\partial_u \partial_v r|, |\partial_u \partial_u r|, |\partial_v \partial_v r|$. Through {\color{black}(\ref{Omega upper})} and Proposition \ref{1stderivofOmega}, we bound $|\O|, |\partial_u \log\O|, |\partial_v \log\O|$.

\noindent To control $\partial_v \partial_u \O$, we use
$$\partial_v\partial_u (\O^2)=2\partial_v (\O\cdot\partial_u \O)=2\O\cdot\partial_v\partial_u\O+2\partial_v\O\cdot\partial_u\O,$$
which implies
\begin{equation}\label{Omega 1}
\O\cdot\partial_v \partial_u \O=\f12\cdot\partial_v\partial_u (\O^2)-\partial_v \O\cdot\partial_u \O=\f12\cdot\partial_v\partial_u (\O^2)-\O^2\cdot\partial_v \log \O\cdot\partial_u \log \O.
\end{equation}
For $\partial_v \partial_u (\O^2)$, we have
$$\f{\partial_v\partial_u{\color{black}(\O^2)}}{\O^2}-\f{\partial_u {\color{black}(\O^2)}\cdot \partial_v \log{\color{black}(\O^2)}}{\O^2}=\partial_v \f{\partial_u \O^2}{\O^2}=\partial_v\partial_u \log{\color{black}(\O^2)},$$
which gives
\begin{equation}\label{Omega 2}
\f12\cdot{\partial_v\partial_u{\color{black}(\O^2)}}-\f12\cdot{\partial_u {\color{black}(\O^2)}\cdot \partial_v \log{\color{black}(\O^2)}}=\f12\cdot\O^2\cdot\partial_v\partial_u \log{\color{black}(\O^2)}.
\end{equation}
{\color{black}Combining (\ref{Omega 1}) and (\ref{Omega 2})}, we have
\begin{equation}\label{Omega 3}
\O\cdot\partial_v \partial_u \O=\f12\cdot \partial_u (\O^2)\cdot \partial_v \log(\O^2)-\O^2\cdot\partial_v \log\O\cdot \partial_u \log\O+\f12\cdot \O^2\cdot\partial_v\partial_u \log{\color{black}(\O^2)}.
\end{equation}
With estimates derived in (\ref{Omega upper}), {\color{black}Propositions \ref{1stderivofOmega} and \ref{2ndderivofOmega}},
we hence obtain polynomial upper bound for $|\O\cdot\partial_v\partial_u\O|$. \\

\noindent The last step is to derive upper bound for $1/\O^2$. This is equivalent to deriving lower bound for $\O^{2}$.  Here we appeal to Theorem \ref{prop1.2}. From (\ref{eqn Omega}){\color{black}:}
\begin{equation*}
r^2\partial_u \partial_v \log \O=\partial_u r \partial_v r+\f14\O^2-r^2 \partial_u \phi \partial_v \phi,
\end{equation*}
we have
\begin{equation*}
\begin{split}
\log \Omega(U_0, V_0){\color{black} =}&{\color{black}\log\Omega(U_0,V')+\log\Omega(U',V_0)-\log\Omega(U',V')}\\
+&\int_{U'}^{U_0}\int_{V'}^{V_0}\bigg(\f{\partial_u r \partial_v r}{r^2}+\f{\Omega^2}{4r^2}-\partial_u \phi \partial_v \phi \bigg)(u,v)dudv.
\end{split}
\end{equation*}
{\color{black} With the bounds in Proposition \ref{Prop 4.1} for $\partial_u r, \partial_v r$, and the estimate in (\ref{Omega upper}) for $\O^2$, we have}
\begin{equation*}
\begin{split}
&|\log \Omega(U_0, V_0)|\\
\leq&{\color{black}|\log\Omega(U_0,V')|+|\log\Omega(U',V_0)|+|\log\Omega(U',V')|}+\int_{U'}^{U_0}\int_{V'}^{V_0}\f{\tilde{D}^2}{r^4}(u,v)dudv\\
\leq&\tilde{C}\log\f{1}{r}(U_0, V_0)\leq\log\f{1}{r^{\tilde{C}}}(U_0, V_0),
\end{split}
\end{equation*}
where $\tilde{D}$ and $\tilde{C}$ are some uniform positive constants depending on initial data. This implies
$$\log\Omega(U_0,V_0)\geq \log r^{\tilde{C}}(U_0, V_0),$$
and together with (\ref{Omega upper}) we have
$$ r^{2\tilde{C}}(U_0, V_0)   \leq \Omega^2(U_0, V_0)\leq  \f{D}{r(U_0, V_0)}, \quad \mbox{and}$$
$$\f{r(U_0, V_0)}{D}\leq \Omega^{-2}(U_0, V_0)\leq r^{-2\tilde{C}}(U_0, V_0).$$
Putting all these estimates together, we hence conclude that: \\

\noindent Theorem \ref{thm1.1}. For spacetime solutions to (\ref{ES}) under spherical symmetry, at any point {\color{black}$(u,v)\in\mathcal{T}$ and $(u,v)$ close to $\mathcal{B}$}, there exists a positive number $N$ (depending on {\color{black} the} initial data at an earlier time), such that
$$\f{1}{r(u,v)^6} \lesssim R^{\alpha\beta\gamma\delta}R_{\alpha\beta\gamma\delta}\lesssim \f{1}{r(u,v)^N}.$$

\begin{remark}
Integrating $\partial_u \phi$ respect to $u$, using Proposition \ref{Prop 4.1}, we also conclude $|\phi|\ls |\log r|$ {\color{black}uniformly in the region $\{u \leq U_0, v \leq V_0\}$ close enough to $P$}.

\end{remark}

}

\section{A Case Study: {\color{black}Perturbations} of {\color{black}the} Schwarzschild Solution}\label{A Case Study}

\subsection{Initial Data}

Denote $o_0(1)$ to be a small {\color{black}number depending} only on initial data.

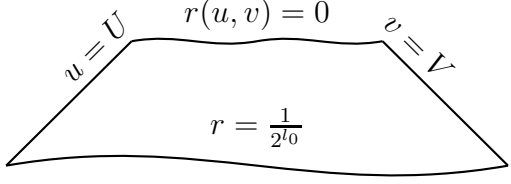
\begin{minipage}[!t]{0.55\textwidth}
\begin{tikzpicture}[scale=0.55]
\draw [white](-3, 0)-- node[midway, sloped, above,black]{$r(u,v)=0$}(3, 0);
\draw [white](-4, -1)-- node[midway, sloped, above,black]{$u=U$}(-3, 0);
\draw [white](3, 0)-- node[midway, sloped, above,black]{$v=V$}(4, -1);
\draw [thick] (-3,0) to [out=10,in=-170] (0,0);
\draw [thick] (0,0) to [out=10,in=-170] (3,0);
\draw [thick] (-6,-3)--(-3,0);
\draw [thick] (3,0)--(6,-3);
\draw [white](-5, -2.8)-- node[midway, sloped, above,black]{$r=\f{1}{2^{l_0}}$}(5, -2.8);
\draw [thick] (-6, -3) to [out=10,in=-170] (6,-3);
\end{tikzpicture}
\end{minipage}
\begin{minipage}[!t]{0.4\textwidth}
We consider the {\color{black}open}trapezoid region $T_0$ below.
For initial data along $r=1/2^{l_0}$, we prescribe
\end{minipage}
\hspace{0.05\textwidth}

$$\partial_v r+\f{M}{r}=o_0(1)\cdot \f{M}{r}, \quad \quad \partial_u r+\f{M}{r}=o_0(1)\cdot \f{M}{r},$$
$$\O^2=\f{2M}{r}+o_0(1)\cdot \f{M}{r},$$
$$|\partial_u \phi|\leq o_0(1)\cdot \f{1}{r^2}, \quad \quad |\partial_v \phi|\leq o_0(1)\cdot \f{1}{r^2}.$$

\subsection{Bootstrap Assumptions}
Fix {\color{black}a} positive parameter $M$. We choose the following bootstrap assumptions
\begin{equation}\label{bootstrap rv}
\partial_v r+\f{M}{r}=o(1)\cdot \f{M}{r},
\end{equation}

\begin{equation}\label{bootstrap ru}
\partial_u r+\f{M}{r}=o(1)\cdot \f{M}{r},
\end{equation}

\begin{equation}\label{bootstrap Omega}
\O^2\leq \f{3M}{r},
\end{equation}

\subsection{{\color{black}Improving} the Estimates}
\subsubsection{Estimates for $\partial_u r$ and $\partial_v r$}

With (\ref{eqn r}), we have
$$\partial_u (r\partial_v r)=-\f14\O^2.$$
This gives
$$(r\partial_v r)(U,v)-(r\partial_v r)(U',v)=\int^U_{U'}-\f14\O^2(u,v) du,$$
and
\begin{equation*}
\begin{split}
|(r\partial_v r)(U,v)-(r\partial_v r)(U',v)|\leq& \int^U_{U'}\f14\O^2(u,v) du\leq \int^{r(U,v)}_{r(U',v)}\f14\f{\O^2}{\partial_u r}(u,v)dr\\
\lesssim&r(U', v)-r(U,v).
\end{split}
\end{equation*}
Together with $(r\partial_v r)(U',v)=-M+o_0(1)$, we have
$$|(r\partial_v r)(U,v)+M|\leq r(U', v)-r(U,v)+o_0(1).$$
Pick $r(U',v)$ and $o_0(1)$ sufficient small, we have
$$|(r\partial_v r)(U,v)+M|\leq r(U', v)-r(U,v)+o_0(1){\color{black}\leq 2^{-l_0}+o_0(1)} <\f12 o(1).$$
This improves (\ref{bootstrap rv}). Similarly, we also could improve (\ref{bootstrap ru}).

\subsubsection{Estimates for $\Omega^2$}\label{Omega^2 1st}
With (\ref{eqn u}) we have
$$\partial_u[\f{\O^2}{-4\partial_u r}]=[\f{\O^2}{-4\partial_u r}]\cdot \f{r(\partial_u \phi)^2}{\partial_u r}.$$
This implies
$$\f{\O^2}{-4\partial_u r}(U,v)=\f{\O^2}{-4\partial_u r}(U',v)\cdot \exp[\int^U_{U'}\f{r(\partial_u \phi)^2}{\partial_u r}(u,v)du].$$
Hence,
$$\f{\O^2}{-4\partial_u r}(U,v)\leq \f{\O^2}{-4\partial_u r}(U',v).$$
With {\color{black}the information of the} initial data and estimates for $\partial_u r$, we have derived
$$\O^2 (U,v)\leq \f{2.5M}{r}.$$
This improves (\ref{bootstrap Omega}).

\subsubsection{Estimates for $\partial_u\phi$ and $\partial_v \phi$}.
We consider the spacetime region (in $T_0$) below. Let $1\ll l\ll n$.

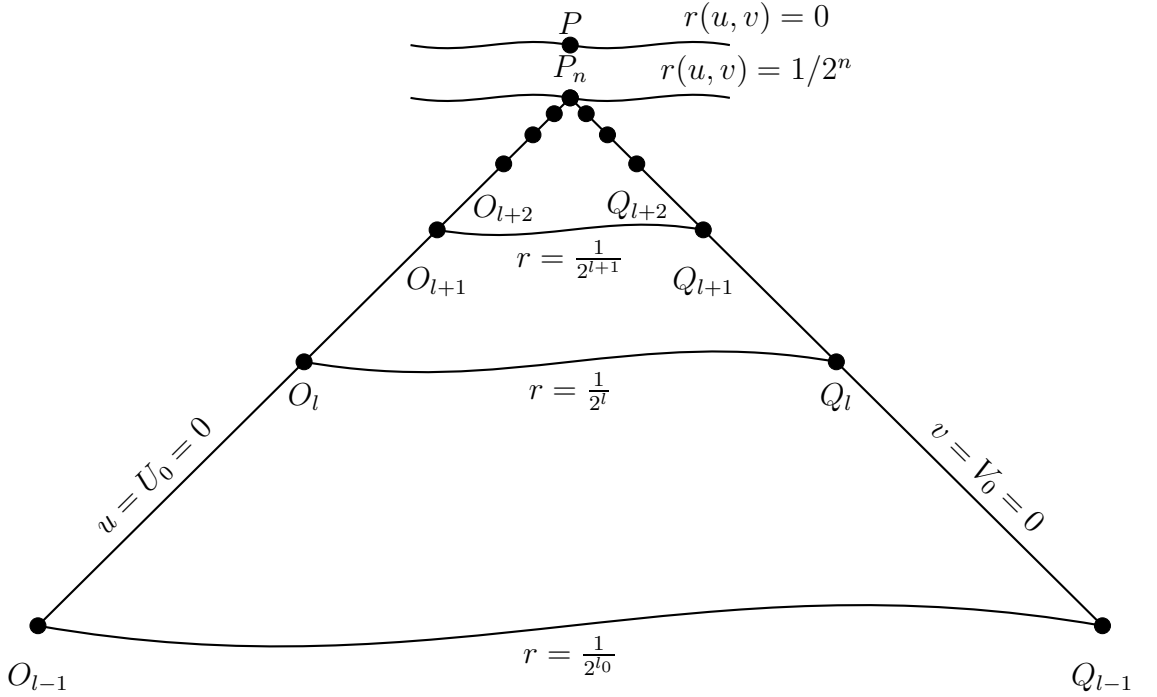
\begin{minipage}[!t]{0.4\textwidth}
\begin{tikzpicture}[scale=0.7]
\draw [white](4, 1)-- node[midway, sloped, above,black]{$r(u,v)=0$}(3, 1);
\draw [white](-2, 1)-- node[midway, sloped, above,black]{$P$}(2, 1);
\draw [white](-2, 0.1)-- node[midway, sloped, above,black]{$P_n$}(2, 0.1);
\draw [white](-2.51, -3)-- node[midway, sloped, below,black]{$O_{l+1}$}(-2.5, -3);
\draw [white](-6, -5.2)-- node[midway, sloped, below,black]{$O_l$}(-4, -5.2);
\draw [white](-11, -10.5)-- node[midway, sloped, below,black]{$O_{l-1}$}(-9, -10.5);
\draw [white](-1.26, -1.63)-- node[midway, sloped, below,black]{$O_{l+2}$}(-1.25, -1.63);
\draw [white](2.51, -3)-- node[midway, sloped, below,black]{$Q_{l+1}$}(2.5, -3);
\draw [white](6, -5.2)-- node[midway, sloped, below,black]{$Q_l$}(4, -5.2);
\draw [white](11, -10.5)-- node[midway, sloped, below,black]{$Q_{l-1}$}(9, -10.5);
\draw [white](1.28, -1.57)-- node[midway, sloped, below,black]{$Q_{l+2}$}(1.25, -1.57);
\draw [white](4, 0)-- node[midway, sloped, above,black]{$r(u,v)=1/2^n$}(3, 0);

\draw [white](-5, -5)-- node[midway, sloped, above, black]{$u=U_0=0$}(-10,-10);
\draw [white](5, -5)-- node[midway, sloped, above, black]{$v=V_0=0$}(10,-10);
\draw[thick] (0,0)--(-10,-10);
\draw[thick] (0,0)--(10, -10);
\draw [thick] (-3,0) to [out=-10,in=170] (0,0);
\draw [thick] (0,0) to [out=-10,in=170] (3,0);
\draw [thick] (-3,1) to [out=-10,in=170] (0,1);
\draw [thick] (0,1) to [out=-10,in=170] (3,1);
\draw [white](-10, -10)-- node[midway, sloped, below,black]{$r=\f{1}{2^{l_0}}$}(10, -10);
\draw [white](-5, -5)-- node[midway, sloped, below,black]{$r=\f{1}{2^{l}}$}(5, -5);
\draw [white](-2.5, -2.5)-- node[midway, sloped, below,black]{$r=\f{1}{2^{{\color{black}l+1}}}$}(2.5, -2.5);
\draw[thick] (-10, -10) to [out=-10, in=170] (10, -10);
\draw[thick] (-5, -5) to [out=-10, in=170] (5, -5);
\draw[thick](-2.5, -2.5) to [out=-10, in=170] (2.5, -2.5);
\draw[fill] (0,0) circle [radius=0.15];
\draw[fill] (0,1) circle [radius=0.15];
\draw[fill] (0,0) circle [radius=0.15];
\draw[fill] (-10,-10) circle [radius=0.15];
\draw[fill] (-5,-5) circle [radius=0.15];
\draw[fill] (-2.5,-2.5) circle [radius=0.15];
\draw[fill] (-1.25,-1.25) circle [radius=0.15];
\draw[fill] (-0.7,-0.7) circle [radius=0.15];
\draw[fill] (-0.3,-0.3) circle [radius=0.15];
\draw[fill] (10,-10) circle [radius=0.15];
\draw[fill] (5,-5) circle [radius=0.15];
\draw[fill] (2.5,-2.5) circle [radius=0.15];
\draw[fill] (1.25,-1.25) circle [radius=0.15];
\draw[fill] (0.7,-0.7) circle [radius=0.15];
\draw[fill] (0.3,-0.3) circle [radius=0.15];
\end{tikzpicture}
\end{minipage}
\hspace{0.05\textwidth}
\begin{minipage}[!t]{0.6\textwidth}
\end{minipage}

\noindent For different constant {\color{black}$r$-level} sets $\{L_r\}$. Let
$$\Psi(r)=\max\{\sup_{P\in L_r} |C_2\cdot r\partial_u \phi|(P), \sup_{Q\in L_r} |C_1\cdot r\partial_v \phi|(Q)\}.$$
At $P$, we have
$$-r\partial_u r(P)=C_1>0, \quad -r\partial_v r(P)=C_2>0.$$
By (\ref{bootstrap ru}) and (\ref{bootstrap rv}), we have
$$|C_1-M|\leq o(1), \quad \quad |C_2-M|\leq o(1).$$
Then from
$${\color{black} \partial_{u}(r\partial_v \phi)=-r_v \partial_u\phi}$$
we have
\begin{equation}
\begin{split}
|C_1\cdot r\partial_v \phi|(P_n)\leq& o_0(1)+ \int_{u(Q_{l-1})}^{u(P_n)} -r_v |C_1\cdot\partial_u\phi| \, du\\
\leq& o_0(1)+ \int_{u(Q_{l-1})}^{u(P_n)} -r_u\cdot \f{r \partial_v r}{r \partial_u r} |C_1\cdot\partial_u\phi| \, du\\
=& o_0(1)+\int_{r(Q_{l-1})}^{r(P_n)} -\f{r \partial_v r}{r\partial_u r}\cdot\f{C_1}{C_2}\cdot\f{1}{r}\cdot |C_2\cdot r\partial_u\phi| \, dr\\
=& o_0(1)+\int_{r(Q_{l-1})}^{r(P_n)} -\f{1+O(r^{\f{1}{100}})}{r}\cdot |C_2\cdot r\partial_u\phi| \, dr \quad \quad (\mbox{use Proposition }\ref{improved ru rv})
\end{split}
\end{equation}
Similarly, we have
$$|C_2\cdot r\partial_u \phi|(P_n)\leq o_0(1)+\int_{r(O_{l-1})}^{r(P_n)} -\f{1+O(r^{\f{1}{100}})}{r}\cdot |C_1\cdot r\partial_v\phi| \, dr.$$
Combining these two inequality together, we have
$$\Psi(2^{-n})\leq o_0(1)+\int_{r=2^{-l+1}}^{r=2^{-n}} -\f{1+O(r^{\f{1}{100}})}{r}\cdot \Psi(r) \, dr.$$
Here $2^{-n}$ could be replaced by any small positive number. Hence it is true that for any small enough $\tilde{r}>0$
$$\Psi(\tilde{r})\leq o_0(1)+\int_{2^{-l+1}}^{\tilde{r}} -\f{1+O(r^{\f{1}{100}})}{r}\cdot \Psi(r) \, dr=o_0(1)+\int^{2^{-l+1}}_{\tilde{r}} \f{1+O(r^{\f{1}{100}})}{r}\cdot \Psi(r) \, dr$$
{\color{black}By} Gr\"onwall's inequality, we have
$$\Psi(\tilde{r})\leq o_0(1)\times e^{\int_{\tilde{r}}^{2^{-l+1}} \f{1+O(r^{\f{1}{100}})}{r}\cdot} dr=o_0(1)\times e^{-\ln \tilde{r}+O(1)}\leq\f{o_0(1)}{\tilde{r}}.$$
This gives
$$\tilde{r}\Psi(\tilde{r})\leq o_{{\color{black}0}}(1) \mbox{ for any } \tilde{r}>0,$$
which further implies
$$r^2|\partial_u \phi|\leq o_{{\color{black}0}}(1), \quad r^2|\partial_v \phi|\leq o_{{\color{black}0}}(1).$$

\subsection{Upper Bounds for Kretschmann Scalar}

We now start to derive an upper bound for Kretschmann scalar. Recall from (\ref{Kretschmann2}) in {\color{black}Section} \ref{Kretschmann}, we have

\begin{equation}\label{Kretschmanneq}
\begin{split}
&R^{\alpha\beta\rho\sigma}R_{\alpha\beta\rho\sigma}\\
=&\f{4}{r^4 \O^8}\bigg(16\cdot(\f{\partial^2 r}{\partial u \partial v})^2\cdot r^2\cdot\O^4+16\cdot \f{\partial^2 r}{\partial u^2}\cdot\f{\partial^2 r}{\partial v^2} \cdot r^2\cdot\O^4 \bigg)\\
&+\f{4}{r^4 \O^8}\bigg(-32\cdot \f{\partial^2 r}{\partial u^2}\cdot\partial_v r \cdot r^2 \cdot \O^3\cdot \partial_v \O -32\cdot\f{\partial^2 r}{\partial v^2}\cdot r^2\cdot\partial_u r \cdot \O^3\cdot \partial_u \O \bigg)\\
&+\f{4}{r^4 \O^8}\bigg( 16\cdot (\partial_v r)^2\cdot (\partial_u r)^2\cdot \O^4+64\cdot \partial_v r\cdot r^2\cdot \partial_u r\cdot \O^2\cdot \partial_u \O\cdot \partial_v\O+8\cdot\partial_v r\cdot \partial_u r\cdot \O^6\bigg)\\
&+\f{4}{r^4 \O^8}\bigg(  16\cdot r^4\cdot (\f{\partial^2 \O}{\partial v \partial u})^2\cdot \O^2-32\cdot r^4\cdot \f{\partial^2 \O}{\partial v \partial u}\cdot \O \cdot \partial_v \O \cdot \partial_u \O \bigg)\\
&+\f{4}{r^4 \O^8}\bigg( 16\cdot r^4\cdot (\partial_v \O)^2\cdot (\partial_u \O)^2+\O^8 \bigg).
\end{split}
\end{equation}

Thus, to obtain an upper bound for $R^{\alpha\beta\rho\sigma}R_{\alpha\beta\rho\sigma}$, it is crucial to derive an upper bound for $\Omega^{-2}$.

\begin{proposition}
Give the prescribed initial data, in the diamond region, we have
$$\Omega^{-1}(U,v)\leq {C}\cdot{r(U,v)^{\f12-o(1)^2}},$$
for some positive constant $C$.
\end{proposition}
\begin{proof}
We revisit subsection \ref{Omega^2 1st}. From (\ref{eqn u}) we have
$$\partial_u[\f{\O^2}{-4\partial_u r}]=[\f{\O^2}{-4\partial_u r}]\cdot \f{r(\partial_u \phi)^2}{\partial_u r}.$$
This implies
\begin{equation}\label{Omega lower bound}
\f{\O^2}{-4\partial_u r}(U,v)=\f{\O^2}{-4\partial_u r}(U',v)\cdot \exp[\int^U_{U'}\f{r(\partial_u \phi)^2}{\partial_u r}(u,v)du].
\end{equation}
Since
$$|\partial_u \phi|\leq o_{{\color{black}0}}(1)/r^2, \quad \mbox{and} \quad \partial_u r+\f{M}{r}=o(1)\cdot \f{M}{r},$$
we have
\begin{equation*}
\begin{split}
\int^U_{U'}\f{r(\partial_u \phi)^2}{\partial_u r}(u,v)du \geq& -o_{{\color{black}0}}(1)^2\cdot\int_{U'}^U \f{1}{Mr^2}(u,v)du\geq -o_{{\color{black}0}}(1)^2\cdot \int_{r(U,v)}^{r(U',v)}\f{1}{M^2 r}(u,v)d r(u,v)\\
\geq& o_{{\color{black}0}}(1)^2\cdot \log r(U,v)\geq \log r(U,v)^{o_{{\color{black}0}}(1)^2}.
\end{split}
\end{equation*}
Using (\ref{Omega lower bound}), we have
\begin{equation*}
\begin{split}
\f{\O^2}{-4\partial_u r}(U,v)=&\f{\O^2}{-4\partial_u r}(U',v)\cdot \exp[\int^U_{U'}\f{r(\partial_u \phi)^2}{\partial_u r}(u,v)du]\\
\geq& \f{\O^2}{-4\partial_u r}(U',v)\cdot r(U,v)^{o_{{\color{black}0}}(1)^2}\\
\geq& \f12\cdot [1+o(1)]\cdot r(U,v)^{o_{{\color{black}0}}(1)^2}.
\end{split}
\end{equation*}
This gives
$$\Omega^2(U,v)\geq -4\partial_u r (U,v)\cdot  \f12\cdot [1+o(1)]\cdot r(U,v)^{o(1)^2} \leq 2M \cdot [1+o(1)]\cdot r(U,v)^{o_{{\color{black}0}}(1)^2-1}.$$
Therefore, we obtain
$$\Omega^{-2}(U,v)\leq \f{1+o(1)}{2M}\cdot {r(U,v)^{1-o_{{\color{black}0}}(1)^2}},$$
and
$$\Omega^{-1}(U,v)\leq {C}\cdot{r(U,v)^{\f12-o_{{\color{black}0}}(1)^2}},$$
for some positive constant $C$.
\end{proof}

We now move on to prove the main theorem of this section.
\begin{theorem}\label{close to Schwarzschild} We consider the {\color{black}open} trapezoid region $T_0$ lying in $\mathcal{T}$ below. 

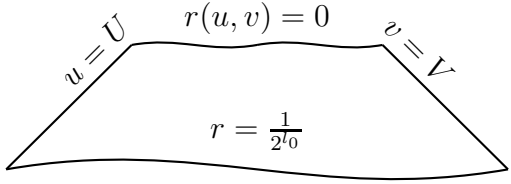
\begin{minipage}[!t]{0.55\textwidth}
\begin{tikzpicture}[scale=0.55]
\draw [white](-3, 0)-- node[midway, sloped, above,black]{$r(u,v)=0$}(3, 0);
\draw [white](-4, -1)-- node[midway, sloped, above,black]{$u=U$}(-3, 0);
\draw [white](3, 0)-- node[midway, sloped, above,black]{$v=V$}(4, -1);
\draw [thick] (-3,0) to [out=10,in=-170] (0,0);
\draw [thick] (0,0) to [out=10,in=-170] (3,0);
\draw [thick] (-6,-3)--(-3,0);
\draw [thick] (3,0)--(6,-3);
\draw [white](-5, -2.8)-- node[midway, sloped, above,black]{$r=\f{1}{2^{l_0}}$}(5, -2.8);
\draw [thick] (-6, -3) to [out=10,in=-170] (6,-3);
\end{tikzpicture}
\end{minipage}
\begin{minipage}[!t]{0.4\textwidth}
For $l_0$ being a large positive constant, we prescribe initial data along $r=1/2^{l_0}$: requiring
\end{minipage}
\hspace{0.05\textwidth}

$$|\partial_v r+\f{M}{r}|\leq o_0(1)\cdot \f{M}{r}, \quad \quad |\partial_u r+\f{M}{r}|\leq o_0(1)\cdot \f{M}{r},$$
$$|\O^2-\f{2M}{r}|\leq o_0(1)\cdot \f{M}{r},$$
$$|\partial_u \phi|\leq o_0(1)\cdot \f{1}{r^2}, \quad \quad |\partial_v \phi|\leq o_0(1)\cdot \f{1}{r^2},$$
where $o_0(1)$ is a small positive number depending on initial data. Then for the dynamical spacetime solutions of (\ref{ES}) under spherical symmetry, under the prescribed initial data, in the region above, we have
\begin{equation}
|R^{\alpha\beta\rho\sigma}R_{\alpha\beta\rho\sigma}|\lesssim \f{1}{r^{6+o_{{\color{black}0}}(1)^2}}.
\end{equation}
\end{theorem}

\begin{proof}
Use (\ref{eqn Omega}), we have
$$|\f{\partial_v \O}{\O}|=|\partial_v \log \O|\leq \f{C}{r^2}.$$
This implies
$$|\partial_v \O|\leq\f{C}{r^2}\cdot |\O|\leq \f{C}{r^3}.$$
Similarly, we have
$$|\partial_u \O|\leq \f{C}{r^3}.$$
Since
$$\partial_u \partial_v \log \O=\partial_u (\f{\partial_v \O}{\O})=\f{\partial_u \partial_v \O}{\O}-\f{\partial_u\O\cdot \partial_v \O}{\O^2},$$
we have
$$\f{\partial^2 \O}{\partial v \partial u}\cdot \O-\partial_v \O\cdot\partial_u \O=\O^2 \partial_u \partial_v \log\O.  $$

We then bound the last two lines in (\ref{Kretschmanneq})
\begin{equation*}
\begin{split}
&|\f{4}{r^4 \O^8}\bigg(  16\cdot r^4\cdot (\f{\partial^2 \O}{\partial v \partial u})^2\cdot \O^2-32\cdot r^4\cdot \f{\partial^2 \O}{\partial v \partial u}\cdot \O \cdot \partial_v \O \cdot \partial_u \O \bigg)\\
&+\f{4}{r^4 \O^8}\bigg( 16\cdot r^4\cdot (\partial_v \O)^2\cdot (\partial_u \O)^2+\O^8 \bigg)|\\
=&\f{4}{r^4 \O^8}\cdot 16\cdot r^4\cdot \bigg(\f{\partial^2 \O}{\partial v \partial u}\cdot \O-\partial_v \O\cdot\partial_u \O\bigg)^2+\f{4}{r^4}\\
=&\f{64}{\O^4}\cdot(\partial_u \partial_v \log \O)^2+\f{4}{r^4}\\
\leq& \f{C}{r^8}\cdot r^{2-o_{{\color{black}0}}(1)^2}+\f{4}{r^4}\\
\leq& \f{C}{r^{6+o_{{\color{black}0}}(1)^2}}.
\end{split}
\end{equation*}

We control the third line in (\ref{Kretschmanneq})
\begin{equation*}
\begin{split}
&|\f{4}{r^4 \O^8}\bigg( 16\cdot (\partial_v r)^2\cdot (\partial_u r)^2\cdot \O^4+64\cdot \partial_v r\cdot r^2\cdot \partial_u r\cdot \O^2\cdot \partial_u \O\cdot \partial_v\O+8\cdot\partial_v r\cdot \partial_u r\cdot \O^6\bigg)|\\
\leq& |\f{64}{r^4 \O^4}\cdot (\partial_v r)^2\cdot (\partial_u r)^2+\f{256}{r^2\O^4}\cdot\partial_v r\cdot \partial_u r\cdot \partial_u \log\O\cdot \partial_v \log\O+\f{32}{r^4 \O^2}\cdot\partial_v r\cdot\partial_u r| \\
\leq&\f{C}{r^{6+o_{{\color{black}0}}(1)^2}}.
\end{split}
\end{equation*}

Finally, we bound the first two lines in (\ref{Kretschmanneq}).

Note that (\ref{eqn u}) is equivalent to
$$-\f{2}{\O^3}\cdot\partial_u \log \O\cdot \partial_u r+\f{\partial^2 r}{\partial u^2}=-r(\partial_u \phi)^2.$$
Hence,
$$\f{\partial^2 r}{\partial u^2}=-r(\partial_u \phi)^2+\f{2}{\O^3}\cdot\partial_u \log \O\cdot \partial_u r .$$
Similarly, we have
$$\f{\partial^2 r}{\partial v^2}=-r(\partial_v \phi)^2+\f{2}{\O^3}\cdot\partial_v \log \O\cdot \partial_v r .$$
From (\ref{eqn r}), we have
$$\partial_{u}\partial_v r=-\f{\partial_u r \partial_v r}{r}-\f{1}{4r}\O^2.$$
Using all the estimates derived above, we thus conclude
$$|\f{\partial^2 r}{\partial u^2}|\leq \f{C}{r^3}, \quad |\f{\partial^2 r}{\partial v^2}|\leq \f{C}{r^3}, \quad |\partial_u \partial_v r|\leq \f{C}{r^3}.$$

We have for the first line in (\ref{Kretschmanneq}):
\begin{equation*}
\begin{split}
&|\f{4}{r^4 \O^8}\bigg(16\cdot(\f{\partial^2 r}{\partial u \partial v})^2\cdot r^2\cdot\O^4+16\cdot \f{\partial^2 r}{\partial u^2}\cdot\f{\partial^2 r}{\partial v^2} \cdot r^2\cdot\O^4 \bigg)|\leq \f{C}{r^{6+o_{{\color{black}0}}(1)^2}}.
\end{split}
\end{equation*}

The second line in (\ref{Kretschmanneq}) obeys
\begin{equation*}
\begin{split}
&|\f{4}{r^4 \O^8}\bigg(-32\cdot \f{\partial^2 r}{\partial u^2}\cdot\partial_v r \cdot r^2 \cdot \O^3\cdot \partial_v \O -32\cdot\f{\partial^2 r}{\partial v^2}\cdot r^2\cdot\partial_u r \cdot \O^3\cdot \partial_u \O \bigg)|\\
=&|\f{4}{r^4 \O^8}\bigg(-32\cdot \f{\partial^2 r}{\partial u^2}\cdot\partial_v r \cdot r^2 \cdot \O^4\cdot \partial_v \log\O -32\cdot\f{\partial^2 r}{\partial v^2}\cdot r^2\cdot\partial_u r \cdot \O^4\cdot \partial_u \log \O \bigg)|\\
\leq& \f{C}{r^{6+o_{{\color{black}0}}(1)^2}}.
\end{split}
\end{equation*}

Therefore, we have proved the main theorem of this section.
\end{proof}

\end{document}